\documentclass[11pt,reqno]{amsart}
\usepackage[top=3cm, bottom=2.5cm, left=2.5cm, right=2.5cm]{geometry}      
\usepackage[utf8]{inputenc}
\usepackage[T1]{fontenc}

\usepackage[english]{babel}
\usepackage{amsmath,mathtools} 
\usepackage{amsfonts} 
\usepackage{amssymb} 
\usepackage{graphicx} 
\usepackage[all]{xy} 
\usepackage[linesnumbered,ruled,vlined]{algorithm2e} 
\usepackage{hyperref} 
\usepackage{amsthm} 
\usepackage{enumitem}  
\usepackage{emptypage} 
\usepackage{xcolor} 
\usepackage{ifthen} 
\usepackage{pgfplots} 
\usepackage{tikz-3dplot} 
\usetikzlibrary{shapes.geometric} 
\tdplotsetmaincoords{60}{130} 
\usepackage{verbatim} 
\usepackage{todonotes}
\tikzset{
	every picture/.style={
		scale=0.4,
		every node/.style={scale=0.8}
	}
}

\theoremstyle{plain} 
\newtheorem{theorem}{Theorem}[section] 
\newtheorem{proposition}[theorem]{Proposition} 
\newtheorem{lemma}[theorem]{Lemma} 
\newtheorem{corollary}[theorem]{Corollary} 

\theoremstyle{definition} 
\newtheorem{definition}[theorem]{Definition} 
\newtheorem{remark}[theorem]{Remark} 
\newtheorem{example}[theorem]{Example} 
\numberwithin{equation}{section}
\newcommand{\N}{\mathbb{N}} 
\newcommand{\Z}{\mathbb{Z}} 
\newcommand{\Chi}{\mathcal{X}} 
\newcommand{\mat}[2]{\mathcal{M}_{#1}(#2)} 
\newcommand{\mindeg}[1]{\mathrm{mindeg}(#1)} 
\newcommand{\K}{\mathbb{K}} 
\newcommand{\T}{\mathbb{T}} 
\newcommand{\Oid}{\mathcal{O}} 
\newcommand{\dO}[1]{\ifthenelse{\equal{#1}{1}}{\partial \Oid}{\partial^{#1} \Oid}} 
\newcommand{\ind}[2]{\mathrm{ind}_{#1}(#2)} 
\newcommand{\supp}[1]{\mathrm{Supp}(#1)} 
\newcommand{\J}{\mathcal{J}_{\dO{1}}} 
\newcommand{\C}[1]{\mathcal{C}(#1)} 
\newcommand{\h}[1]{\mathrm{Ht}(#1)} 
\newcommand{\TT}[1]{\mathcal{T}#1} 
\newcommand{\nf}[1]{\mathrm{Nf}(#1)} 
\newcommand{\lt}[2]{\mathrm{lt}_{#1}(#2)} 
\newcommand{\inid}[2]{\mathrm{in}_{#1}(#2)} 
\newcommand{\lcm}[1]{\mathrm{lcm}(#1)} 
\newcommand{\reg}[1]{\mathrm{reg}(#1)} 

\begin{document}


\title{Homogeneous Border Bases on Infinite Order Ideals}
\author[C. Bertone]{Cristina Bertone}
\address{Dipartimento di Matematica \lq\lq G.~Peano\rq\rq\\ Universit\`a degli Studi di Torino}
\email{\href{mailto:cristina.bertone@unito.it}{cristina.bertone@unito.it}}
\urladdr{\url{https://sites.google.com/view/cristinabertone/}}
\author[S. Bovero]{Sofia Bovero}
\address{Department of Molecular Biotechnology and Health Sciences, Universit\`a degli Studi di Torino}
\email{\href{mailto:sofia.bovero@edu.unito.it}{sofia.bovero@edu.unito.it}}


\begin{abstract}
	Border bases are traditionally restricted to 0-dimensional ideals due to the finiteness of the underlying order ideal. In this paper we extend the theory to homogeneous ideals of positive Krull dimension by introducing homogeneous border bases, defined relative to an infinite order ideal. Moreover, we provide two characterizations of these bases: one via border reductors and, most notably, one in terms of formal multiplication matrices. 
    Although the latter condition a priori requires verification in infinitely many degrees, we prove that it is sufficient to check only finitely many of them, thereby obtaining an effective criterion.
\end{abstract}

 \keywords{Border basis, homogeneous ideal, positive Krull dimension}
  \subjclass{13P10, 14Q20}
\maketitle

\section{Introduction}

Since their first appearances in \cite{marinari1991grobner, marinari1993groebner, mourrain1999new}, border bases over a finite order ideal have garnered the interest of mathematicians due to their several features. Indeed, border bases were studied from a numerical point of view due to their stability with respect to perturbations of the coefficients \cite{MollerStetter, stetter2004, ABT}. Furthermore, several researchers have investigated border bases from the algebraic point of view, highlighting their appropriate structure for symbolic computation; see, for instance, \cite{KKR2005}, \cite[Section~6.4]{kreuzer2005computational}, \cite{MourrainTreb05}. 
The algebraic and computational features of border bases provide a useful framework for investigating the Hilbert scheme. Thanks to the fact that the family of border bases over a given finite order ideal parameterizes an open subset of a Hilbert scheme of 0-dimensional schemes, many authors have used border bases to  uncover properties of the Hilbert scheme. A non-exhaustive list of works employing border bases in this context includes: \cite{kreuzer2008deformations}, \cite{kr2011}, \cite{Lederer}, \cite{huibregtse2011some}, \cite{HuibregtseElemComp}, \cite{HuibregtseElemComp2023}. Some applications exist also in areas other than algebraic geometry or symbolic computation in commutative algebra; for a recent example, see \cite{RodSatt}.

The use of border bases is unfortunately confined to ideals of Krull dimension 0 in the polynomial ring $\K[x_1,\dots,x_n]$, given that they are defined relative to an order ideal of finite cardinality. Geometrically, the zero-set of such an ideal is a finite set of points in the affine space $\mathbb A^n_{\K}$. 

The goal of the present paper is to give a generalization of the notion of border basis to a homogeneous setting, preserving the algebraic structure that allows explicit computations, starting from an infinite order ideal. These homogeneous border bases generate homogeneous ideals in $\K[x_0,\dots,x_n]$ with positive Krull dimension, whose zero set are algebraic sets in $\mathbb P^n_{\K}$. 
Our setting is close to that of \cite{MourrainTreb} and \cite[Section~6]{bertone2024cohen}.

To our knowledge, \cite{MourrainTreb} is the first paper where the notion of border basis for a non-homogeneous ideal is given by dropping the hypothesis that the order ideal is finite. Also in \cite{chen2007border} the authors consider infinite order ideals, but they use a monomial order so they are dealing with Gr\"obner bases.  The main result of \cite{MourrainTreb} is Algorithm~4.1, that takes as input a (non-homogeneous) ideal in the polynomial ring and outputs an order ideal and a border basis on it generating the given ideal. 
 In the present paper, besides working in a homogeneous setting, we move in the opposite direction. Starting from an infinite order ideal, we  define the notions of homogeneous border prebasis and basis (Definitions~\ref{def: prebasis} and~\ref{def: border basis}), and characterize when a border prebasis is a basis (Theorems~\ref{char_rewrite_relations} and~\ref{char comm mat}).
 In our setting it is very important that the notion of border prebasis comes along with a \emph{border reduction structure}, that needs to be carefully chosen in order to have a polynomial reduction relation which is Noetherian and confluent (see Definition~\ref{def: border reduction relation} and Section~\ref{subsec Noeth and confl})

More recently, in \cite[Section~6]{bertone2024cohen}, working on homogeneous Artinian ideals, the authors used marked bases over a quasi-stable ideal to construct polynomials belonging to the given ideal whose head terms are exactly the border of an infinite order ideal. By these new polynomials, the authors investigate some features of the points of the scheme parametrized by marked bases, i.e. identifying which ideals are complete intersections. The definition of polynomials with head terms in the border starting from a marked basis was expected to be a winning strategy, given that, for 0-dimensional ideals in the affine space, marked bases over a quasi-stable ideal and border bases are closely related  \cite{bertone2022close}.

The construction given in \cite[Section~6]{bertone2024cohen} with homogeneous polynomials inspired us, with the aim of finding a notion of homogeneous border basis for non-Artinian homogeneous ideals.

 The content of the paper is organized in the following way.
 In Section~\ref{sec notations} we set notations and recall some properties of order ideals and their border. In Section~\ref{sec border reduction} we define border reduction structures and border prebases (Definitions~\ref{def: border reduction structure} and~\ref{def: prebasis}), which together allow to have a polynomial reduction relation (Definition~\ref{def: border reduction relation}) which uses \emph{border reductors} (Definition~\ref{def border reductors}). By making a reasonable assumption on the labeling of the terms in the order ideal, we obtain that the reduction is Noetherian and confluent (Section~\ref{subsec Noeth and confl}). In Section~\ref{sec basis}, we define the notion of homogeneous border basis on an order ideal, prove its existence and uniqueness (Definition~\ref{def: border basis} and Propositions~\ref{prop: eu basis} and~\ref{prop: extension groebner}). In Theorem~\ref{char_rewrite_relations} we also provide a first criterion to determine whether a border prebasis is a basis by border reductors. In Section~\ref{sec matrices} we give another criterion for border bases that uses \emph{formal multiplication matrices}, which generalizes the well-known analogous criterion for classical border bases; see for instance \cite[Proposition~16]{kehrein2005characterizations}.
Due to the fact that we are considering a homogeneous border basis on an infinite order ideal, both characterizations of Theorems~\ref{char_rewrite_relations} and~\ref{char comm mat} are not effective, since there are conditions to check for every $d\geq 0$. In Section~\ref{sec finite conditions} we address this flaw using Gotzmann's Theorem, in a  way similar to \cite{MourrainTreb}.
Along the paper, several examples are given; in particular, in the final Example~\ref{ex finale}, we give an explicit set of polynomial conditions on the coefficients of a border prebasis that guarantee it is a basis. We conclude by sketching some future developments and applications.

\section{Notations and Background}\label{sec notations}

Let $\K$ be a field. We consider the polynomial ring $P=\K[x_0,\dots,x_n]$. 
A {\em term} in $P$ is a power product $ x_0^ {\alpha_0}\cdots x_n^{\alpha_n}$, where $(\alpha_0,\ldots,\alpha_n) \in \mathbb N^{n+1}$. We denote the set of terms in $P$ with $\T$; for every $L\subseteq\{x_0,\dots,x_n\}$ we write $\T(L)$ for the set of the terms such that $\alpha_j=0$ for all $x_j\notin L$.\\
We use the standard grading on $P$, i.e.~$\deg(x_j)=1$, for all $j\in \{0,\dots,n\}$. Thus, $\deg( x_0^ {\alpha_0}\cdots x_n^{\alpha_n})=\sum \alpha_i$.

If $S$ is a  subset of $P$, we denote by $\langle S\rangle$ the $\K$-vector space generated by $S$, and by $(S)$ the ideal generated by $S$ in $P$. Furthermore, for every $t\in \mathbb Z$, we denote by $S_{t}$ the subset of $S$ consisting of homogeneous elements of degree $t$. Similarly, we denote by $S_{\geqslant m}$ and $S_{\leqslant m}$ the direct sums $S=\bigoplus_{t \geqslant m }S_{t}$ and $S=\bigoplus_{t \leqslant m }S_{t}$.
Throughout this paper, we consider 
homogeneous ideals $J$ of $P$, so that $J=\bigoplus_{t \in \mathbb N}J_{t}$.
 
For every element $f\in P$, we denote by $\supp{f}$ the set of terms appearing in $f$ with non-zero coefficient.


We now recall the definitions of order ideal, border of an order ideal, and properties of the latter.

\begin{definition}{\cite[Definitions~6.4.3 and~6.4.4]{kreuzer2005computational}} \label{def: order ideal, border, border closure}
	A non-empty set of terms $\Oid \subseteq \T$ is called an \emph{order ideal} if it is closed under forming divisors, i.e. if
	\[
	\forall \tau \in \Oid, \forall \tau' \in \T,\; \tau' \mid \tau \Rightarrow \tau' \in \Oid.
	\]
	The \emph{border} of $\Oid$ is the set of terms
	\[
	\dO{1} \coloneqq \T_1\cdot \Oid \setminus \Oid = (x_0\Oid \cup \dots \cup x_n\Oid) \setminus \Oid
	\]
	and the \emph{first border closure} of $\Oid$ is the set $\overline{\dO{1}} \coloneqq \Oid \cup \dO{1}$. \\
	Since $\overline{\dO{1}}$ is an order ideal too, for every integer $k \geq 0$ we can inductively define the \emph{$k$-th border} of $\Oid$ by the rule
	\[
	\dO{0} \coloneqq \Oid \text{ and } \dO{k} \coloneqq \partial (\overline{\dO{k-1}})
	\]
	and the \emph{$k$-th border closure} of $\Oid$ by
	\[
	\overline{\dO{0}} \coloneqq \dO{0} = \Oid \text{ and } \overline{\dO{k}} = \overline{\dO{k-1}} \cup \dO{k}.
	\]
\end{definition}

The following proposition contains some consequences of Definition~\ref{def: order ideal, border, border closure}.
\begin{proposition}{\cite[Proposition~6.4.6]{kreuzer2005computational}} \label{prop: basic prop borders}
	Let $\Oid \subseteq \T$ be an order ideal.
	\begin{enumerate}[label=(\roman*)]
		\item\label{prop: basic prop borders 1} For every $k \geq 0$, we have a disjoint union $\overline{\dO{k}} = \bigcup_{i=0}^k \dO{i}$. Consequently, we have a disjoint union $\T = \bigcup_{i=0}^\infty \dO{i}$. 
		\item\label{prop: basic prop borders 2} For every $k \geq 1$, we have $\dO{k} = \T_k\cdot \Oid \setminus \T_{<k}\cdot \Oid$. 
		\item\label{prop: basic prop borders 3} A term $\tau \in \T$ is divisible by a term in $\dO{1}$ if and only if $\tau \in \T \setminus \Oid$. 
	\end{enumerate}
\end{proposition}

The above partition of $\T$ allows us to define a ``distance'' between a term and an order ideal.
\begin{definition}{\cite[Definition~6.4.7]{kreuzer2005computational}} \label{def: index}
	Let $\Oid \subseteq \T$ be an order ideal.
	\begin{enumerate}[label=\alph*)]
		\item For every $\tau \in \T$, the unique number $k \in \N$ such that $\tau \in \dO{k}$ is called the \emph{index} of $\tau$ with respect to $\Oid$ and is denoted by $\ind{\Oid}{\tau}$.
		\item For a polynomial $f \in P \setminus \left\{0\right\}$, we define the \emph{index} of $f$ with respect to $\Oid$ by $\ind{\Oid}{f} = \max \left\{ \ind{\Oid}{\tau} \mid \tau \in \supp{f} \right\}$.
	\end{enumerate}
\end{definition}

Note how the two concepts of index and degree are complementing one another.
\begin{proposition}{\cite[Proposition~6.4.8]{kreuzer2005computational}} \label{prop: basic prop index}
	Let $\Oid \subseteq \T$ be an order ideal.
	\begin{enumerate}[label=(\roman*)]
		\item\label{prop: basic prop index 1} For a term $\tau \in \T$, the number $k = \ind{\Oid}{\tau}$ is the smallest natural number such that $\tau = \tau'\tau''$ with $\tau' \in \T_k$ and $\tau'' \in \Oid$. 
		\item\label{prop: basic prop index 2} Given two terms $\tau,\tau' \in \T$, we have $\ind{\Oid}{\tau\tau'} \leq \deg(\tau) + \ind{\Oid}{\tau'}$. 
		\item\label{prop: basic prop index 3} For non-zero polynomials $f,g \in P$ such that $f+g \neq 0$, we have
		\[
		\ind{\Oid}{f+g} \leq \max \left\{ \ind{\Oid}{f}, \ind{\Oid}{g} \right\}. 
		\] 
		\item\label{prop: basic prop index 4} For non-zero polynomials $f,g \in P$, we have
		\[
		\ind{\Oid}{fg} \leq \min \left\{ \deg(f) + \ind{\Oid}{g}, \deg(g) + \ind{\Oid}{f} \right\}. 
		\] 
	\end{enumerate}
\end{proposition}

Although the partial ordering on $\T$ defined by the index appears similar to a term ordering, it has a serious drawback: this ordering is incompatible with multiplication, i.e. $\ind{\Oid}{\tau} \leq \ind{\Oid}{\tau'}$ does not, in general, imply $\ind{\Oid}{\tau\tau''} \leq \ind{\Oid}{\tau'\tau''}$.
\begin{example}\label{ex:drawback index}
	Consider the order ideal $\Oid = \T(x) \cup \left\{ y, xy, x^2y, y^2, xy^2, x^2y^2 \right\} \subseteq \K[x,y]$ whose border is $\dO{1} = \left\{ x^ky \mid k \geq 3 \right\} \cup \left\{ y^3, xy^3, x^2y^3, x^3y^2 \right\}$. The following diagram illustrates the situation. The point $(r,s)$ of the diagram corresponds to the term $x^ry^s$. We use the symbol {\tikz \node[fill,circle,inner sep=2pt] {};} for terms in $\Oid$, {\tikz \node[draw,circle,inner sep=2pt] {};} for terms in $\dO{1}$, and {\tikz {\draw (-0.15,-0.15) -- (0.15,0.15);\draw (-0.15,0.15) -- (0.15,-0.15);}} for terms in $\dO{2}$.
	\begin{center}
		\begin{tikzpicture}
			\draw[->] (0,0) -- (7.5,0) node[right] {\textit{x}};
			\draw[->] (0,0) -- (0,5) node[above] {\textit{y}};
			\node[below left] at (0,0) {\text{1}};
			\foreach \x/\y in {0/0, 0/1, 1/1, 2/1, 0/2, 1/2, 2/2, 1/0, 2/0, 3/0, 4/0, 5/0, 6/0, 7/0} {
				\node[fill,circle,inner sep=2pt] at (\x,\y) {};
			}
			\foreach \x/\y in {0/3, 1/3, 2/3, 3/2, 3/1, 4/1, 5/1, 6/1, 7/1} {
				\node[draw,circle,inner sep=2pt] at (\x,\y) {};
			}
			\foreach \x/\y in {0/4, 1/4, 2/4, 3/3, 4/2, 5/2, 6/2, 7/2} {
				\draw (\x-0.15,\y-0.15) -- (\x+0.15,\y+0.15);
				\draw (\x-0.15,\y+0.15) -- (\x+0.15,\y-0.15);
			}
		\end{tikzpicture}
	\end{center}
	Multiplying the terms on both sides of $\ind{\Oid}{x^2y^2} < \ind{\Oid}{y^3}$ by $x^2$, we get $\ind{\Oid}{x^2 \cdot x^2y^2} > \ind{\Oid}{x^2 \cdot y^3}$. Similarly, multiplying the terms on both sides of $\ind{\Oid}{y^3} = \ind{\Oid}{x^2y^3}$ by $x$, we get $\ind{\Oid}{x \cdot y^3} < \ind{\Oid}{x \cdot x^2y^3}$.
\end{example}

\section{Polynomial Border Reduction}\label{sec border reduction}
\begin{definition}\label{def: border reduction structure}
	Given an order ideal $\Oid \subseteq \T$, let the terms of the border $\dO{1}$ 
    be ordered in an arbitrary way and labeled coherently, i.e. for every 
    $\sigma_i, \sigma_j\in \dO{1}$, if $i<j$ then 
    $\sigma_i$ precedes 
    $\sigma_j$. \\ 
	The \emph{border reduction structure} $\J$ is the 3-uple 
	\[
    \J \coloneqq \left(\dO{1} , \, \mathcal{L} \coloneqq \left\{\mathcal{L}_\sigma \mid \sigma \in \dO{1}\right\}, \, \mathcal{T} \coloneqq \left\{\mathcal{T}_\sigma \mid \sigma \in \dO{1}\right\} \right)
	\] 
	where:
	\begin{itemize}
		\item for every 
        $\sigma \in \dO{1}$, $\mathcal{L}_{\sigma} \coloneqq \Oid_{\deg(\sigma)}$ is the \emph{tail set} of $\sigma$;
		\item for every $\sigma_i \in \dO{1}$, $\mathcal{T}_{\sigma_i} \coloneqq \left\{ \eta \in \T : \forall j>i, \, \sigma_j \nmid \eta\sigma_i \right\}$ is the \emph{multiplicative set} of $\sigma_i$. 
	\end{itemize}
	The elements of $\mathcal{T}_{\sigma_i}$ are referred to as the \emph{multiplicative terms} for $\sigma_i$.

    Given a term $\sigma \in \dO{1}$, we define the \emph{cone} of $\sigma$ with respect to the border reduction structure $\J$ as the set  $\mathcal{C}_{\J}(\sigma) \coloneqq \left\{\eta\sigma \mid \eta \in \mathcal{T}_\sigma \right\}$.
\end{definition}

\begin{remark}
We observe 
that two reduction structures on the same set $\dO{1}$ differ if and only if there is at least one term in $\dO{1}$ having different sets of multiplicative terms in the two reduction structures (and hence different cones in the two reduction structures). The latter happens if and only if the reduction structures are given by two different orders of the terms in $\dO{1}$, i.e. two different labellings for them. This fact is highlighted in the following example.
\end{remark}

\begin{example}\label{es: cones}
	Consider the order ideal $\Oid = \T(x) \cup y\cdot \T(x) \cup \left\{ y^2, xy^2, x^2y^2, y^3 \right\} \subseteq \K[x,y]$, whose border is $\dO{1} = \left\{ y^4, xy^3, x^2y^3 \right\} \cup \left\{ x^ky^2 \mid k \geq 3 \right\}$. 
	We consider two different border reduction structures, with the terms of the border ordered increasingly by degree but in different ways. \\
	For the first border reduction structure $\J = \left(\dO{1},\, \mathcal{L},\, \mathcal{T}\right)$, we order the border in the following way:
    \[
	\sigma_1=y^4, \; \sigma_2=xy^3, \; \sigma_3=x^2y^3, \; \sigma_4=x^3y^2, \; \sigma_\ell=x^{\ell-1}y^2 \quad \ell\geq 5.
	\]
	We obtain that the multiplicative sets are $\mathcal{T}_\sigma = \T(y)$ for every $\sigma \in \dO{1}$. In this case the cones are:
    \begin{gather*}
        \mathcal{C}_{\J}(y^4)= \{y^k \mid k \geq 4 \}, \quad
        \mathcal{C}_{\J}(xy^3)= \{xy^k \mid k \geq 3 \}, \quad
        \mathcal{C}_{\J}(x^2y^3)= \{x^2y^k \mid k \geq 3 \}, \\
        \mathcal{C}_{\J}(x^\ell y^2)= \{x^\ell y^k \mid k \geq 2 \} \; \text{ for } \ell \geq 3.
    \end{gather*}
    The following diagram illustrates the situation. The blue arrows represent the cone of the term from which they originate.
	
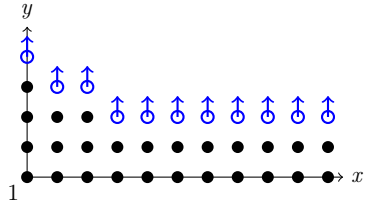
\begin{center}
		\begin{tikzpicture}
			\draw[->] (0,0) -- (10.5,0) node[right] {\textit{x}};
			\draw[->] (0,0) -- (0,5) node[above] {\textit{y}};
			\node[below left] at (0,0) {\text{1}};
			\foreach \x/\y in {0/0, 1/0, 2/0, 3/0, 4/0, 5/0, 6/0, 7/0, 8/0, 9/0, 10/0, 0/1, 1/1, 2/1, 3/1, 4/1, 5/1, 6/1, 7/1, 8/1, 9/1, 10/1, 0/2, 1/2, 2/2, 0/3} {
				\node[fill,circle,inner sep=2pt] at (\x,\y) {};
			}
			\foreach \x/\y in {0/4, 1/3, 2/3, 3/2, 4/2, 5/2, 6/2, 7/2, 8/2, 9/2, 10/2} {
				\node[draw,circle,blue,thick,inner sep=2pt] at (\x,\y) {};
			}
			\foreach \x/\y/\dx/\dy in {0/4/0/0.7,
				1/3/0/0.7,
				2/3/0/0.7, 
				3/2/0/0.7, 
				4/2/0/0.7, 5/2/0/0.7, 6/2/0/0.7, 7/2/0/0.7, 8/2/0/0.7, 9/2/0/0.7, 10/2/0/0.7
			} {
				\draw[blue, thick, ->] (\x,\y) -- (\x+\dx,\y+\dy);
			}
		\end{tikzpicture}
	\end{center}
	We now consider the border reduction structure $\J' = \left(\dO{1},\, \mathcal{L}',\, \mathcal{T}'\right)$, with the terms of the border ordered in the following way:
    \[
	\sigma_1=y^4, \; \sigma_2=xy^3, \; \sigma_3=x^3y^2, \; \sigma_4=x^2y^3, \; \sigma_\ell=x^{\ell-1}y^2 \quad \ell\geq 5,
	\]
	and we obtain that the multiplicative sets are 
	\[
	\mathcal{T}'_{x^3y^2} = \{1\}, \quad
	\mathcal{T}'_{x^2y^3} = \T(y) \cup x\cdot\T(y), \quad
	\mathcal{T}'_\sigma = \T(y) \; \text{ for every } \sigma \neq x^3y^2, x^2y^3.
	\]
    In this second case the cones are:
    \begin{gather*}
        \mathcal{C}_{\J'}(y^4)= \{y^k \mid k \geq 4 \}, \quad
        \mathcal{C}_{\J'}(xy^3)= \{xy^k \mid k \geq 3 \}, \quad
        \mathcal{C}_{\J'}(x^3y^2)= \{x^3y^2\}, \\
        \mathcal{C}_{\J'}(x^2y^3)= \{x^2y^k \mid k \geq 3 \} \cup \{x^3y^k \mid k \geq 3 \}, \quad
        \mathcal{C}_{\J'}(x^\ell y^2)= \{x^\ell y^k \mid k \geq 2 \} \; \text{ for } \ell \geq 4.
    \end{gather*}
	The following diagram illustrates the situation.
	\begin{center}
		\begin{tikzpicture}
			\draw[->] (0,0) -- (10.5,0) node[right] {\textit{x}};
			\draw[->] (0,0) -- (0,5) node[above] {\textit{y}};
			\node[below left] at (0,0) {\text{1}};
			\foreach \x/\y in {0/0, 1/0, 2/0, 3/0, 4/0, 5/0, 6/0, 7/0, 8/0, 9/0, 10/0, 0/1, 1/1, 2/1, 3/1, 4/1, 5/1, 6/1, 7/1, 8/1, 9/1, 10/1, 0/2, 1/2, 2/2, 0/3} {
				\node[fill,circle,inner sep=2pt] at (\x,\y) {};
			}
			\foreach \x/\y in {0/4, 1/3, 2/3, 3/2, 4/2, 5/2, 6/2, 7/2, 8/2, 9/2, 10/2} {
				\node[draw,circle,blue,thick,inner sep=2pt] at (\x,\y) {};
			}
			\foreach \x/\y/\dx/\dy in {0/4/0/0.7,
				1/3/0/0.7,
				2/3/0/0.7, 3/3/0/0.7,
				4/2/0/0.7, 5/2/0/0.7, 6/2/0/0.7, 7/2/0/0.7, 8/2/0/0.7, 9/2/0/0.7, 10/2/0/0.7
			} {
				\draw[blue, thick, ->] (\x,\y) -- (\x+\dx,\y+\dy);
			}
			\foreach \x/\y/\dx/\dy in {2/3/1/0} {
				\draw[blue, thick, -] (\x,\y) -- (\x+\dx,\y+\dy);
			}
		\end{tikzpicture}
	\end{center}
\end{example}

When the border reduction structure is clear from the context, we simply write $\C{\sigma}$ instead of $\mathcal{C}_{\J}(\sigma)$.

\begin{remark}
	Note that a border reduction structure $\J$ is a reduction structure as defined in \cite[Definition~3.1]{ceria2019general} if we relax the hypothesis on the cardinality of the set $M$, 
    namely we allow this set to be infinite. Indeed $\bigcup_{
    \sigma \in \dO{1}} \C{\sigma}  \allowbreak = (\dO{1}) 
    $ (as shown in \cite[Lemma~13.2.ii]{ceria2019general}, whose proof remains valid also for an infinite order ideal). Furthermore, for all 
    $\sigma\in \dO{1}$, we have that $\mathcal{T}_\sigma$ is an order ideal (as shown in \cite[Lemma~13.2.i]{ceria2019general}, whose proof remains valid for infinite order ideals too). Finally, $\mathcal{L}_\sigma$ is a finite subset of $\T \setminus \C{\sigma}$. \\
	Moreover, here $\J$ differs from the classical border reduction structure for a finite order ideal, as defined in \cite[Definition~11]{bertone2022close}, not only in having an infinite border, but also in the definition of $\mathcal{L}_
    \sigma$, which consists solely of terms of degree 
    $\deg(\sigma)$.
\end{remark}

The following lemma is proved as part  of \cite[Proposition~2]{bertone2022close}, which remains valid for infinite order ideals and for any ordering of the border.
\begin{lemma}\label{lem: disjoint cones}
	Let $\Oid\subseteq \T$ be an order ideal, and label the terms in $\dO{1}$ as in Definition~\ref{def: border reduction structure}. A border reduction structure $\J$ has disjoint cones, i.e.
    $\C{\sigma_i} \cap \C{\sigma_j} = \emptyset$
	for every pair of distinct elements 
    $\sigma_i,\sigma_j\in \dO{1}$.
\end{lemma}

\begin{definition}\label{def: prebasis}
	Let $\Oid$ be an order ideal. A set of polynomials $G = \left\{ 
    g_\sigma\right\}_{\sigma \in \dO{1}}$ is called a \emph{homogeneous $\dO{1}$-prebasis}, or \emph{homogeneous border prebasis on $\Oid$}, if the polynomials have the form
	\[
    g_\sigma = \sigma - \sum_{\tau \in \Oid_{\deg(\sigma)}} c_{\sigma \tau} \tau
	\]
	with 
    $c_{\sigma \tau}\in \K$, for every 
    $\sigma\in \dO{1}$.
\end{definition}

We can  associate to a $\dO{1}$-prebasis $G$, once a border reduction structure is fixed, a reduction procedure $\to^+_G$.
\begin{definition}\label{def: border reduction relation}
	Let $\Oid$ be an order ideal, let the terms in $\dO{1}$ be ordered as in Definition~\ref{def: border reduction structure}, and let $\J$ be a border reduction structure. Consider  a homogeneous $\dO{1}$-prebasis $G$. The \emph{border reduction relation} associated to $G$ and $\J$ is the transitive closure $\to^+_{G,\J}$ of the relation $\to_{G,\J}$ on $P$ that is defined in the following way:\\
    for $f,h \in P$,  $f$ is in relation with $h$ (and we write $f \to_{G,\J} h$) when there exist terms 
    $\beta\in \supp{f}$ and 
    $\sigma \in \dO{1}$ such that 
    $\beta = \eta \sigma \in \C{\sigma}$
    and
	\[
    h=f-c\eta g_\sigma,
	\]
	where $c \in \K$ is the coefficient of 
    $\beta$ in $f$. 
	When $f \to^+_{G,\J} h$, we say that $f$ \emph{reduces to} $h$ with respect to $G$.
\end{definition}
When there is no ambiguity, we write $\to_{G}$ and $\to^+_{G}$ instead of $\to_{G,\J}$ and $\to^+_{G,\J}$.

\begin{remark}
A homogeneous $\dO{1}$-prebasis is made of marked polynomials in the sense of \cite{ReevesS}, and 
is also a marked set on a border reduction structure $\J$ as defined in \cite[Definition~4.2]{ceria2019general}. Indeed, for every 
    $\sigma\in \dO{1}$ there is a unique $g_\sigma \in G$ with a unique term of the support belonging to $(\dO{1})$, which is $\sigma$. We define 
    $\h{g_\sigma}=\sigma$ and we call it \emph{head term} (or \emph{marked term}) of $g_\sigma$. Observe that 
    $\supp{g_\sigma - \sigma} \subseteq \Oid_{\deg(\sigma)}=\mathcal{L}_\sigma$. The polynomial $g_\sigma-\sigma$ is the \emph{tail} of $g_\sigma$.

    However, the reduction relation of Definition~\ref{def: border reduction relation} is not the one of \cite{ReevesS}. Indeed, we define the present reduction relation by using a border reduction structure $\J$, as in  \cite[page~105]{ceria2019general}. 
\end{remark}

A term 
$\beta\in (\dO{1})$ is reduced using the term 
$\sigma_i \in \dO{1}$ with $i = \max\left\{ \ell \mid 
\sigma_\ell \in \dO{1} \allowbreak \text{ divides } \beta \right\}$, that is 
$\sigma_i$ such that 
$\beta\in \C{\sigma_i}$.
Such 
$\sigma_i$ is unique by Lemma~\ref{lem: disjoint cones}.
\begin{definition}{\cite[Definition~4.2]{ceria2019general}} \label{def border reductors}Let $\Oid$ be an order ideal, $\J$ a border reduction structure  and $G$ be a $\dO{1}$-prebasis.
	We define the set of all \emph{border reductors} with respect to the border reduction procedure $\to^+_{G,\J}$ as $\mathcal{T}_{\J}G \coloneqq \left\{ 
    \eta g_\sigma \mid g_\sigma \in G, \, \eta \in \mathcal{T}_\sigma\right\}$. 
	For each $d \in \mathbb{N}$, we also define its degree-$d$ part: ${\mathcal{T}_{\J}G}_d \coloneqq \mathcal{T}_{\J}G \cap P_d = \left\{ 
    \eta g_\sigma \in \mathcal{T}_{\J}G \,\colon \allowbreak \deg(\eta)+\deg(\sigma)=d \right\}$.
\end{definition}

For the same order ideal and border prebasis, different border reduction structures
yeld to different border reductors.
\begin{example}\label{es: border reductors}
	Consider the order ideal $\Oid$ and the two border reduction structures $\J$, $\J'$ in Example~\ref{es: cones}. Let $G$ be the homogeneous $\dO{1}$-prebasis given by the polynomials
	\begin{gather}\label{eq:es border reductors}
		g_{y^4} = y^4, \quad
		g_{xy^3} = xy^3, \quad
		g_{x^2y^3} = x^2y^3 + x^5, \quad 
		g_{x^3y^2} = x^3y^2 + x^5, \quad
		g_{x^ky^2} = x^ky^2 \; \text{ for } k \geq 4.
	\end{gather}
    Then $\mathcal{T}_{\J}G \neq \mathcal{T}_{\J'}G$. Indeed:
    \[
    {\mathcal{T}_{\J}G}_6 = \left\{ y^2g_{y^4}, \; y^2g_{xy^3}, \; yg_{x^2y^3}, \; yg_{x^3y^2}, \; g_{x^4y^2}  \right\}
    \]
	and
    \[
    {\mathcal{T}_{\J'}G}_6 = \left\{ y^2g_{y^4}, \; y^2g_{xy^3}, \; yg_{x^2y^3}, \; xg_{x^2y^3}, \; g_{x^4y^2} \right\}.
    \]
\end{example}
When no ambiguity arises, we write $\TT{G}$ and $\TT{G}_d$ for $\mathcal{T}_{\J}G$ and ${\mathcal{T}_{\J}G}_d$.

We observe that $\langle \TT{G} \rangle \subseteq (G)$ and $\langle \TT{G}_d \rangle \subseteq (G)_d$.

Let us now state some useful properties of the border reduction procedure.
\begin{lemma}\label{lem: prop reduction}  Let $\Oid$ be an  order ideal, $\J$ a border reduction structure and $G$ be a homogeneous $\dO{1}$-prebasis.
	Consider a polynomial $f \in P$ of degree $d$, and let $h \in P$ such that $f \to^+_G h$.\\
		The polynomial $f-h$ belongs to $\langle \TT{G} \rangle$,  more precisely $f=\sum_{j=1}^s c_{j}\eta_jg_{\sigma_{i_j}}+h$ with 
        \begin{equation}\label{eq:TGwriting}
        f\to_G f-c_{1}\eta_1g_{\sigma_{i_1}}\to_G f-c_{1}\eta_1g_{\sigma_{i_1}}-c_{2}\eta_2g_{\sigma_{i_2}}\to_G\cdots\to_G f-(\sum_{j=1}^s c_{j}\eta_jg_{\sigma_{i_j}}),
        \end{equation}
        where $\eta_jg_{\sigma_{i_j}}\in \TT{G}$, $c_j$ suitable elements of $\K$ and $h=f-(\sum_{j=1}^s c_{j}\eta_jg_{\sigma_{i_j}})$.
		Furthermore,  $\deg(h) \leq \deg(f)$.
        
		Moreover, if $f$ is homogeneous then $f-h \in \langle \TT{G}_d \rangle$, the polynomials $\eta_jg_{\sigma_{i_j}}$ in \eqref{eq:TGwriting}  all belong to $\TT{G}_d$ and $h$ is homogeneous of degree $d$.
	\end{lemma}
\begin{proof}
	The fact that $f-h\in \langle \TT{G}\rangle$ and the existence of the writing \eqref{eq:TGwriting} follow directly from the definition of the border reduction procedure.\\
	Since every term 
    $\beta\in \supp{f}$ satisfies 
    $\deg(\beta)\leq \deg(f)=d$, the border reduction procedure on $f$ involves only elements 
    $\eta g_\sigma \in \TT{G}$ whose degree is at most $d$. Consequently, since $G$ consists of homogeneous polynomials, the support of $h$ cannot contain any term of degree greater than $d$.

    Assume now that $f$ is homogeneous. Then the polynomials $\eta g_\sigma$ appearing in \eqref{eq:TGwriting} are of degree~$d$, hence they belong to $\TT{G}_d$. As a consequence, the support of $h$ must contain only terms of degree~$d$.			
\end{proof}

\begin{definition}{\cite[Definition~6.1]{ceria2019general}} \label{def: TTGrep} Let $\Oid$ be an  order ideal, $\J$ a border reduction structure and $G$ be a homogeneous $\dO{1}$-prebasis. Let $f = \sum_{j=1}^{r} c_j{\eta_j}g_{\sigma_{i_j}}$, where $c_j \in \K$ and each ${\eta_j}g_{\sigma_{i_j}}$ is a distinct element of $\TT{G}$. More precisely, we require that if $g_{\sigma_{i_j}} = g_{\sigma_{i_k}}$ then ${\eta_j} \neq {\eta_k}$. In this case, we say that the writing $\sum_{j=1}^{r} c_j{\eta_j}g_{\sigma_{i_j}}$ is a \emph{representation of f by $\TT{G}$}. 
	If, moreover, there exists $n \in \mathbb{N}$ such that for every $j=1,\dots,r$ it holds $\ind{\Oid}{{\eta_j}\sigma_{i_j}} \leq n$, then we say that the representation of $f$ is an \emph{$n$-lower representation} (or \emph{$n$-LRep} for short).\\ 
	Finally, if all the elements ${\eta_j} g_{\sigma_{i_j}}$ appearing in the representation belong to $\TT{G}_d$ for some fixed degree~$d$, we say that this is a \emph{representation by $\TT{G}_d$}.
\end{definition}

\begin{definition}{\cite[page~185]{bertone2022close}}
	Let $\Oid$ be an order ideal and let $G$ be a homogeneous $\dO{1}$-prebasis, and $f \in P$ a polynomial. A \emph{$(\dO{1})$-reduced form modulo $(G)$ of $f$} is a polynomial $h \in P$ such that $f-h \in (G)$ and $\supp{h} \subseteq \Oid$. 
	If there exists a unique $(\dO{1})$-reduced form modulo $(G)$ of $f$, then it is called \emph{$(\dO{1})$-normal form modulo $(G)$ of $f$} and is denoted by $\nf{f}$.
\end{definition}

We remark that if $f \in \langle\Oid\rangle$, then a $(\dO{1})$-reduced form modulo $(G)$ of $f$ is $f$ itself. In this case we say that $f$ is \emph{reduced with respect to $(\dO{1})$} or is a \emph{$(\dO{1})$-remainder}.


\subsection{Noetherianity and Confluency}\label{subsec Noeth and confl}

Let $\Oid$ be an  order ideal, $\J$ a border reduction structure and $G$ be a homogeneous $\dO{1}$-prebasis. The reduction relation $\to^+_G$ is called \emph{Noetherian} if there is no infinite reduction chain $f_1 \to_G f_2 \to_G f_3 \to_G \cdots$. The relation $\to^+_G$ is called \emph{confluent} if for every polynomial $f \in P$ there exists only one $(\dO{1})$-reduced form $h$ modulo $(G)$ such that $f \to^+_G h$.

Disjoint cones alone are not sufficient to ensure Noetherianity, as the following example shows.
\begin{example}\label{es: non noeth.}
	Consider the order ideal $\Oid = \T(x) \cup \left\{ y,xy,x^2y,y^2,y^3 \right\} \subseteq \K[x,y]$, whose border is $\dO{1} = \left\{x^ky \mid k \geq 3 \right\} \cup \left\{ xy^2,x^2y^2,xy^3,y^4 \right\}$. We consider the border reduction structure $\J$ with the terms of the border ordered in the following way:
	\[
	x^4y, \; x^5y, \; x^6y, \; \dots, \; x^3y, \; x^2y^2, \; xy^3, \; y^4, \; xy^2
	\]
	and the homogeneous $\dO{1}$-prebasis $G$ given by the polynomials
	\begin{gather*}
		g_{x^ky}=x^ky \; \text{ for } k \geq 3, \quad
		g_{x^2y^2}=x^2y^2, \quad 
		g_{xy^3}=xy^3, \quad 
		g_{y^4}=y^4, \quad
		g_{xy^2}= xy^2 - x^2y - y^3.
	\end{gather*}
	We get the following multiplicative sets:
	\begin{gather*}
		\mathcal{T}_{x^ky} = \emptyset \; \text{ for } k \geq 4, \quad
		\mathcal{T}_{x^3y} = \T(x), \quad
		\mathcal{T}_{x^2y^2} = \emptyset, \quad 
		\mathcal{T}_{xy^3} = \emptyset, \quad
		\mathcal{T}_{y^4} = \T(y), \quad
		\mathcal{T}_{xy^2} = \T.
	\end{gather*}
	The following diagram illustrates the situation. The elements of the border with an empty cone are not colored in blue.
	
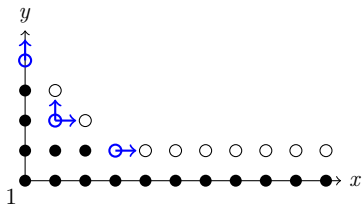
\begin{center}
		\begin{tikzpicture}
			\draw[->] (0,0) -- (10.5,0) node[right] {\textit{x}};
			\draw[->] (0,0) -- (0,5) node[above] {\textit{y}};
			\node[below left] at (0,0) {\text{1}};
			\foreach \x/\y in {0/0, 1/0, 2/0, 3/0, 4/0, 5/0, 6/0, 7/0, 8/0, 9/0, 10/0, 0/1, 1/1, 2/1, 0/2, 0/3} {
				\node[fill,circle,inner sep=2pt] at (\x,\y) {};
			}
			\foreach \x/\y in {4/1, 5/1, 6/1, 7/1, 8/1, 9/1, 10/1, 2/2, 1/3} {
				\node[draw,circle,inner sep=2pt] at (\x,\y) {};
			}
			\foreach \x/\y in {3/1, 1/2, 0/4} {
				\node[draw,circle,blue,thick,inner sep=2pt] at (\x,\y) {};
			}
			\foreach \x/\y/\dx/\dy in {3/1/0.7/0,
				1/2/0.7/0, 1/2/0/0.7, 
				0/4/0/0.7} {
				\draw[blue, thick, ->] (\x,\y) -- (\x+\dx,\y+\dy);
			}
		\end{tikzpicture}
	\end{center}
	The reduction structure $\J$ has disjoint cones, however this does not imply Noetherianity. Indeed, consider the term $xy^3$: the following chain of reductions can be repeated indefinitely
	\begin{gather*}
		xy^3 \; \to_{G,\J} xy^3 - yg_{xy^2} = x^2y^2 + y^4 \; 
		\to_{G,\J} \; x^2y^2 + y^4 - g_{y^4} = x^2y^2 \; \to_{G,\J} \; \\
		\to_{G,\J} \; x^2y^2 - xg_{xy^2} = x^3y + xy^3 \; 
		\to_{G,\J} \; x^3y + xy^3 - g_{x^3y} = xy^3.
	\end{gather*}
\end{example}

Noetherianity and confluency follow when we order the terms of $\dO{1}$ increasingly by degree (terms of the same degree are ordered arbitrarily).
 
\begin{remark}
 We remark that, with the terms of the border ordered increasingly by degree, a term $\beta \in (\dO{1})$ is reduced using a term $\sigma \in \dO{1}$ of maximal degree dividing $\beta$.\\
 Furthermore, in this setting we always have $\{1\} \subseteq \mathcal{T}_{\sigma_i}$ for every $i$.
    Indeed, if $1 \notin \mathcal{T}_{\sigma_i}$, then there exists $j>i$ such that $\sigma_j\mid 1 \cdot \sigma_i$, which means $\sigma_j \mid \sigma_i$ with $\deg(\sigma_j) \geq \deg(\sigma_i)$. Thus $\sigma_j=\sigma_i$, which is a contradiction.
\end{remark}

The following proposition is a natural generalization of \cite[Proposition~2]{bertone2022close}.
\begin{proposition}\label{prop: noeth, direct sum, confl}
	Given an  order ideal $\Oid$, consider a border reduction structure $\J$ with the terms of the border $\dO{1}$ ordered increasingly by degree and a homogeneous $\dO{1}$-prebasis $G$. 
	\begin{enumerate}[label=(\roman*)]
		\item\label{prop: noeth, direct sum, confl 1} For all $\sigma \in \dO{1}, \tau \in \Oid_{\deg(\sigma)}, \eta \in \mathcal{T}_\sigma$, it holds $\ind{\Oid}{{\eta\sigma}} > \ind{\Oid}{\eta \tau}$.
		\item\label{prop: noeth, direct sum, confl 2} The border reduction relation $\to^+_G$ is Noetherian.
		\item\label{prop: noeth, direct sum, confl 3} It holds that $P = \langle \TT{G} \rangle \oplus \langle \Oid \rangle$, and $P_d = \langle \TT{G}_d \rangle \oplus \langle \Oid_d \rangle$	for every $d \geq 0$.
		\item\label{prop: noeth, direct sum, confl 4} The border reduction relation $\to^+_G$ is confluent.
	\end{enumerate}
\end{proposition}
\begin{proof}
	\begin{enumerate}[label=(\roman*), align=left, labelsep=0.5em, leftmargin=*]
		\item This is proved in \cite[Proposition~2]{bertone2022close} using \cite[Lemma~1]{bertone2022close}, which remain valid for an infinite order ideal.
		\item Consider $f,h \in P$ such that $f \to_G h$. More precisely, $h = f-c \eta g_{\sigma}$ with $\gamma={\eta\sigma} \in \supp{f} \cap \C{\sigma}$, $c \in \K$ the coefficient of $\gamma$ in $f$, and $g_{\sigma}$ the polynomial of $G_{\leq \deg(f)}$ such that $\h{g_{\sigma}}={\sigma}$. Consider a term $\delta \in \supp{h}\setminus\supp{f} \subseteq \supp{\eta g_{\sigma} - \eta\sigma}$, then by~\ref{prop: noeth, direct sum, confl 1} we have $\ind{\Oid}{\delta} < \ind{\Oid}{\gamma}$. Thus, during a reduction step, a term of index $k=\ind{\Oid}{\gamma}$ is removed from $f$ and replaced by terms $\delta$ with index strictly smaller than $k$. Moreover, $\deg(h)\leq \deg(f)$. 
		Therefore, starting from $f$, we reach a $(\dO{1})$-reduced polynomial after finitely many reduction steps: indeed, there are only finitely many terms of degree $d \leq \deg(f)$ with index smaller than or equal to $\ind{\Oid}{f}$. Hence, $\to^+_G$ is Noetherian.
		\item The first expression appears in a more general form in \cite[Corollary~6.3]{ceria2019general}. For sake of clarity, we reproduce the proof here in  our framework. \\
		By the Noetherianity of $\to^+_G$ proved in~\ref{prop: noeth, direct sum, confl 2}, 
        we have $P = \langle\TT{G} \rangle + \langle \Oid \rangle$.
		Furthermore, $\langle \TT{G} \rangle \cap \langle \Oid \rangle = \{0\}$; indeed, suppose $h \in \langle \TT{G} \rangle \cap \langle \Oid \rangle$, $h \neq 0$, and let $h = \sum_{j=1}^{r}c_j{\eta_j}g_{\sigma_{i_j}}$ be a representation of $h$ by $\TT{G}$. We choose an element of maximal index in the nonempty set $\left\{ {\eta_j\sigma_{i_j}} \mid j=1,\dots,r, \; c_j \neq 0 \right\}$. Without loss of generality we can suppose this element of maximal index  is ${\eta_1 \sigma_{i_1}}$. Then $\eta_1\sigma_{i_1}$ appears in the support of $h$: in fact, this term is different from $\eta_j\sigma_{i_j}$ for $j=2,\dots,r$ since by hypothesis $\J$ has disjoint cones, and it does not appear in the support of ${\eta_j}g_{\sigma_{i_j}} - {\eta_j\sigma_{i_j}}$ for some $j=1,\dots,r$ by maximality of the index of $\eta_1\sigma_{i_1}$ and~\ref{prop: noeth, direct sum, confl 1}. We get then a contradiction, since the support of $h$ is contained in $\Oid$. Then $h=0$. \\
		Moreover, if $f \in P_d$ is a homogeneous polynomial of degree $d$ and $f \to^+_G h$, we know by Lemma~\ref{lem: prop reduction} that $f-h \in \langle \TT{G}_d \rangle$ and $h \in P_d$; moreover, if $h$ is reduced, then $h \in \langle \Oid_d \rangle$. Hence, we conclude that $P_d = \langle \TT{G}_d \rangle \oplus \langle \Oid_d \rangle$.
		\item Let $f \in P$, $h_1,h_2 \in \langle\Oid\rangle$ such that $f \to^+_G h_1$ and $f \to^+_G h_2$. Then, $h_2-h_1 \in \langle\Oid\rangle$ and $h_2-h_1 = (f-h_1)-(f-h_2) \in \langle \TT{G} \rangle$. By~\ref{prop: noeth, direct sum, confl 3} we obtain $h_2-h_1=0$, i.e. $h_1=h_2$.
		\qedhere
	\end{enumerate}
\end{proof}

\begin{example}\label{es: noeth.}
    Consider the order ideal $\Oid$ and the border prebasis $G$ in Example~\ref{es: non noeth.}. We take a border reduction structure $\J'$ with the terms of the border ordered increasingly by degree
	\[
	xy^2, \; x^3y, \; x^2y^2, \; xy^3, \; y^4, \; x^4y, \; x^5y, \; x^6y, \; \dots.
	\]
	We get the following multiplicative sets:
	\begin{gather*}
		\mathcal{T}'_{xy^2} = \{1\}, \quad
		\mathcal{T}'_{x^3y} = \{1\}, \quad
		\mathcal{T}'_{x^2y^2} = \{1,x\}, \quad
		\mathcal{T}'_{xy^3} = \{1,x,x^2\}, \\
		\mathcal{T}'_{y^4} = \T(y) \cup x\cdot \T(y) \cup x^2\cdot\T(y) \cup x^3\cdot \T(y), \quad
		\mathcal{T}'_{x^ky} = \T(y) \; \text{ for } k \geq 4.
	\end{gather*}
	The following diagram illustrates the situation.
	
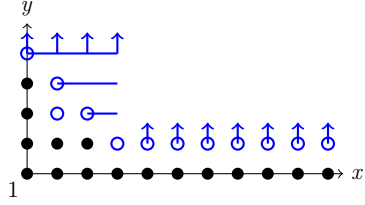
\begin{center}
		\begin{tikzpicture}
			\draw[->] (0,0) -- (10.5,0) node[right] {\textit{x}};
			\draw[->] (0,0) -- (0,5) node[above] {\textit{y}};
			\node[below left] at (0,0) {\text{1}};
			\foreach \x/\y in {0/0, 1/0, 2/0, 3/0, 4/0, 5/0, 6/0, 7/0, 8/0, 9/0, 10/0, 0/1, 1/1, 2/1, 0/2, 0/3} {
				\node[fill,circle,inner sep=2pt] at (\x,\y) {};
			}
			\foreach \x/\y in {4/1, 5/1, 6/1, 7/1, 8/1, 9/1, 10/1, 2/2, 1/3, 3/1, 1/2, 0/4} {
				\node[draw,circle,blue,thick,inner sep=2pt] at (\x,\y) {};
			}
			\foreach \x/\y/\dx/\dy in {2/2/1/0,
				1/3/2/0,
				0/4/3/0} {
				\draw[blue, thick, -] (\x,\y) -- (\x+\dx,\y+\dy);
			}
			\foreach \x/\y/\dx/\dy in {0/4/0/0.7, 1/4/0/0.7, 2/4/0/0.7, 3/4/0/0.7,
				4/1/0/0.7,
				5/1/0/0.7,
				6/1/0/0.7,
				7/1/0/0.7,
				8/1/0/0.7,
				9/1/0/0.7,
				10/1/0/0.7} {
				\draw[blue, thick, ->] (\x,\y) -- (\x+\dx,\y+\dy);
			}
		\end{tikzpicture}
	\end{center}
	Now the chain of reductions starting with $xy^3$ comes to an end: $xy^3 \; \to_{G,\J'} \; xy^3 - g_{xy^3} = 0$.
\end{example}

Noetherianity and confluency lead to some notable consequences.
 
The following corollary appears in a more general form in \cite[Corollary~6.3]{ceria2019general}. For better clarity we include the proof here adapted to our specific setting.
\begin{corollary}\label{cor: eq reduction}
	Given an order ideal $\Oid$, consider a border reduction structure $\J$ with the terms of the border $\dO{1}$ ordered increasingly by degree and a homogeneous $\dO{1}$-prebasis $G$. Let $f \in P$, $h \in \langle \Oid \rangle$. Then the following statements are equivalent:
	\begin{enumerate}[label=(\roman*)]
		\item\label{it:cor eq red 1} $f-h \in \langle \TT{G} \rangle$;
		\item \label{it:cor eq red 2}$f \to^+_G h$;
		\item \label{it:cor eq red 3}$f-h$ has an $\ind{\Oid}{f}$-LRep by $\TT{G}$.
	\end{enumerate}
    Moreover, if $f$ is homogeneous of degree $d$ and $h \in \langle \Oid_d \rangle$ then:
    \[
    f-h \in \langle \TT{G}_d \rangle \quad\Longleftrightarrow\quad f \to^+_G h \quad\Longleftrightarrow\quad f-h \text{ has an } \ind{\Oid}{f} \text{-LRep by } \TT{G}_d.
    \]
\end{corollary}
\begin{proof}
	\begin{enumerate}[label={}, align=left, labelsep=0.5em, leftmargin=*]
		\item[\ref{it:cor eq red 1}$\Rightarrow$\ref{it:cor eq red 2}:] 
        Assume $f-h\in \langle \TT{G}\rangle$. Thanks to Proposition~\ref{prop: noeth, direct sum, confl}, let $f'$ be the unique poynomial in $\langle \Oid\rangle$ such that $f\to^+_G f'$.
        Then we have that $f-h-(f-f')\in \langle \TT{G}\rangle \cap \langle \Oid\rangle$. By Proposition~\ref{prop: noeth, direct sum, confl}.\ref{prop: noeth, direct sum, confl 3}, we obtain $h=f'$.
		\item[\ref{it:cor eq red 2}$\Rightarrow$\ref{it:cor eq red 1}:] This is immediate, since a representation by $\TT{G}$ is given by the steps of the reduction.
		\item[\ref{it:cor eq red 1}$\Rightarrow$\ref{it:cor eq red 3}:] If $f-h \in \langle \TT{G} \rangle$, then $f-h$ has a unique representation $\sum_{j=1}^{r} c_j{\eta_j}g_{\sigma_{i_j}}$ by $\TT{G}$ with $c_j \neq 0$. Let ${\eta_s\sigma_{i_s}}$, for some $s \in \{1,\dots,r\}$, be an element of maximal index in the set $\left\{ {\eta_j \sigma_{i_j}} \mid j=1,\dots,r \right\}$, and let $n \coloneqq \ind{\Oid}{{\eta_s\sigma_{i_s}}}$. Then, 
        we trivially obtain that $f-h$ has an $n$-LRep by $\TT{G}$. Moreover, as 
        in the proof of Proposition~\ref{prop: noeth, direct sum, confl}.\ref{prop: noeth, direct sum, confl 3}, and since $\supp{h} \subseteq \Oid$, ${\eta_s\sigma_{i_s}}$ appears in the support of $f$ and has maximal index between the terms in the support of $\sum_{j=1}^{r} c_j{\eta_j}g_{\sigma_{i_j}}$. Thus, since $f=\sum_{j=1}^{r} c_j{\eta_j}g_{\sigma_{i_j}}+h$ and $\ind{\Oid}{h}=0$, we have $\ind{\Oid}{f} = n$.
		\item[\ref{it:cor eq red 3}$\Rightarrow$\ref{it:cor eq red 1}:] This is immediate from Definition~\ref{def: TTGrep}.
	\end{enumerate}
    Now assume $f \in P_d$ and $h \in \langle \Oid_d \rangle$. The thesis follows by the above equivalences and the final statement of Lemma~\ref{lem: prop reduction}.
\end{proof}

Together with Proposition~\ref{prop: noeth, direct sum, confl}.\ref{prop: noeth, direct sum, confl 3}, the next result generalizes to our context the polynomial representation given by the Border Division Algorithm \cite[Proposition~6.4.11]{kreuzer2005computational}.
\begin{corollary}\label{cor: eq element tg}
	Given an order ideal $\Oid$, consider a border reduction structure $\J$ with the terms of the border $\dO{1}$ ordered increasingly by degree and a homogeneous $\dO{1}$-prebasis $G$.
	\begin{enumerate}[label=(\roman*)]
		\item\label{cor: eq element tg 1} Consider $f \in \langle\TT{G}\rangle$. Then there exist polynomials $g_{\sigma_{i_1}},\dots,g_{\sigma_{i_r}} \in G_{\leq \deg(f)}$ and $q_1,\dots,q_r \in P$ such that 
		\begin{equation}\label{eq:cor eq elem tg 1}
		f = q_1g_{\sigma_{i_1}} + \cdots + q_rg_{\sigma_{i_r}} \; \text{ and } \; \deg(q_j) \leq \ind{\Oid}{f}-1 \text{ whenever } q_j \neq 0.
		\end{equation}
		\item\label{cor: eq element tg 2} Consider $f \in \langle \TT{G}_d \rangle$. Then in the writing \eqref{eq:cor eq elem tg 1}, we have $g_{\sigma_{i_1}},\dots,g_{\sigma_{i_r}} \in G_{\leq d}$ and $q_1,\dots,q_r \in P$ homogeneous of degree $\deg(q_j) = d-\deg(g_{\sigma_{i_j}})$.
	\end{enumerate}
\end{corollary}
\begin{proof}
	\begin{enumerate}[label=(\roman*), align=left, labelsep=0.5em, leftmargin=*]
		\item By 
        Corollary~\ref{cor: eq reduction}.\ref{it:cor eq red 3}, $f$ has an $\ind{\Oid}{f}$-LRep by $\TT{G}$. This means that we can write $f = \sum_{j=1}^{r}c_j{\eta_j}g_{\sigma_{i_j}}$, with $\ind{\Oid}{{\eta_i\sigma_{i_j}}} \leq \ind{\Oid}{f}$ and ${\eta_j}g_{\sigma_{i_j}} \in \TT{G}$. \\
		Since each $\eta_j g_{\sigma_{i_j}}$ is homogeneous, we have that ${\eta_j}g_{\sigma_{i_j}} \in P_{\leq \deg(f)}$. Thus we have $\deg(g_{\sigma_{i_j}}) \leq \deg(f)$. Moreover, since ${\eta_j} \in \mathcal{T}_{\sigma_{i_j}}$ and since $\dO{1}$ is ordered according to ascending degree, by \cite[Lemma~1]{bertone2022close} $\ind{\Oid}{{\eta_j\sigma_{i_j}}} = \deg({\eta_j})+1$, and so $\deg({\eta_j}) \leq \ind{\Oid}{f}-1$.
		\item By the final statement of Corollary~\ref{cor: eq reduction}, $f$ has an $\ind{\Oid}{f}$-LRep by $\TT{G}_d$, with $d=\deg(f)$. The polynomials of $G$ appearing in this representation have degreee $\leq d$, and the statement follows.
		\qedhere
	\end{enumerate}
\end{proof}

\begin{corollary}\label{cor prebasis generates}
Let $\Oid$ be an order ideal and $G$ be a homogeneous $\dO{1}$-prebasis. Then, for every $d\geq 0$, the residue classes of the elements of $\Oid_d$ generate the $\K$-vector space $P_d/(G)_d$.
\end{corollary}
\begin{proof}
Take a border reduction structure $\J$ with the terms of the border $\dO{1}$ ordered increasingly by degree. By Proposition~\ref{prop: noeth, direct sum, confl}.\ref{prop: noeth, direct sum, confl 3}, 
we obtain that 
\[
P_d / \langle\TT{G}_d\rangle \cong \langle\Oid_d\rangle \quad \forall d \geq 0,
\]
and, since $\langle \TT{G}_d \rangle \subseteq (G)_d$, it follows that the residue classes of the elements of $\Oid_d$ generate the $\K$-vector space $P_d/(G)_d$.
\end{proof}

\section{Homogeneous Border Bases}\label{sec basis}

From now on, we always consider an infinite order ideal $\Oid$ 
and a border reduction structure with the terms of $\dO{1}$ ordered increasingly by degree.

In Corollary \ref{cor prebasis generates} we proved that if $G$ is a $\dO{1}$-prebasis, then $P_d/(G)_d$ is generated as a $\K$-vector space by the residue classes of the terms in $\Oid_d$.   However, we note that, in general, this system of generators is not necessarily a basis, as we can see from the following example.
\begin{example}\label{es: generators}
	Let $f = x^3y^3 \in P=\K[x,y]$. Consider the order ideal $\Oid$, the border reduction structures $\J$, $\J'$ and the homogeneous $\dO{1}$-prebasis $G$ in Example~\ref{es: border reductors}. We observed that $\langle \mathcal{T}_{\J}G\rangle \neq \langle \mathcal{T}_{\J'}G\rangle$. We reduce $f$ with respect to both $\J$ and $\J'$: in the first case we find that $f$ can be written as $f = yg_{x^3y^2} - x^5y \in \langle {\mathcal{T}_{\J}G}_6 \rangle \oplus \langle \Oid_6 \rangle$, while in the second case we find $f = xg_{x^2y^3} - x^6 \in \langle \mathcal {\mathcal{T}_{\J'}G}_6 \rangle \oplus \langle \Oid_6 \rangle$.
    Then we have $x^6-x^5y = xg_{x^2y^3} - yg_{x^3y^2} \in (G)_6$, that is $\overline{x^6}-\overline{x^5y} = \overline{0}$ in $P_6/(G)_6$. This means that the residue classes of the elements of $\Oid_6 = \{x^6,x^5y\}$ are not $\K$-linearly independent.
\end{example}


The following definition is the obvious generalization of \cite[Definition~4.3]{ceria2019general} and \cite[Definition~6.4.13]{kreuzer2005computational}.
\begin{definition}\label{def: border basis}
	Let $G$ be a homogeneous $\dO{1}$-prebasis, and let $J \subseteq P$ be a homogeneous ideal containing $G$. The set $G$ is called a \emph{homogeneous $\dO{1}$-basis} of $J$ if one of the following equivalent conditions is satisfied:
	\begin{enumerate}[label=(\arabic*)]
		\item for every $d \geq 0$, the residue classes 
        of the terms in $\Oid_d$ form a $\K$-vector space basis of $P_d / J_d$;
		\item \label{it:def border basis2}for every $d \geq 0$, $P_d = J_d \oplus \langle \Oid_d \rangle$;
		\item\label{it:def border basis3} for every $d \geq 0$, $J_d \cap \langle \Oid_d \rangle = \{0\}$.
	\end{enumerate}
\end{definition}

It is immediate to prove that a homogeneous $\dO{1}$-basis $G$ of $J$ actually generates $J$. The following is a straightforward analogue of \cite[Proposition~4.3.2]{KKR2005}, \cite[Proposition~6.4.15]{kreuzer2005computational}.
\begin{proposition}\label{prop: G generates J}
	Let $G$ be a homogeneous $\dO{1}$-basis of a homogeneous ideal $J \subseteq P$. Then, for every $d \geq 0$ we have $J_d = (G)_d$, and $J=(G)$.
\end{proposition}
\begin{proof}
	 By definition, $(G) \subseteq J$, and thus $(G)_d \subseteq J_d$ for every $d\geq 0$. To prove the converse inclusion, let $f \in J_d$. By Proposition~\ref{prop: noeth, direct sum, confl} and Corollary~\ref{cor: eq element tg}, the polynomial $f$ can be written as 
	\begin{equation}\label{eq:prop base genera}
	f = \sum_{g_\sigma \in G_{\leq d}}q_\sigma g_\sigma + \sum_{\tau \in \Oid_d}c_\tau \tau,
	\end{equation}
	where $c_\tau \in \K$, and $q_\sigma \in P$ is a linear combination of elements in $\mathcal{T}_\sigma$ such that $\deg(q_\sigma g_\sigma)=d$ for every $\sigma$. This implies 
	\[
	\overline{0} = \overline{f} = \sum_{\tau \in \Oid_d}c_\tau \overline{\tau} \quad \text{ in } P_d/J_d.
	\]
	By assumption, the residue classes of $\Oid_d$  are $\K$-linearly independent; hence $c_\tau = 0$ for every $\tau$, and hence \eqref{eq:prop base genera} reduces to
	\[
	f = \sum_{g_\sigma \in G_{\leq d}}q_\sigma g_\sigma \in (G)_d.
	\]
	Thus, $J_d \subseteq (G)_d$, and $J=(G)$ being $J$ a homogeneous ideal.
\end{proof}

\subsection{Existence and Uniqueness}

 
A necessary condition for the existence of a homogeneous $\dO{1}$-basis for the homogeneous ideal $J\subseteq P$ is clearly given by 
\[
|\Oid_d| = \dim_\K(P/J)_d, \, \text{ for every } d \geq 0,
\]
i.e. the Hilbert function of $P/J$ is the cardinality of $\Oid_d$ for every $d$. 
However, our next example shows that this condition is not sufficient.
\begin{example}
	Let $P=\K[x,y]$, and let $J=(xy)$. Therefore $\dim_\K(P/J)_0 = 1$ and $\dim_\K(P/J)_\ell = 2$ for every $\ell \geq 1$. 
	Let $\Oid \subseteq \T$ be the order ideal $\Oid = \T(x) \cup y\cdot\T(x)$. Then $|\Oid_d| = \dim_\K(P/J)_d, \, \forall d \geq 0$. 
	Nevertheless, there is no homogeneous $\dO{1}$-basis contained in $J$; indeed, for $d \geq 2$, the residue classes of the elements of $\Oid_d = \{x^d,x^{d-1}y\}$ are not a $\K$-vector space basis of $(P/J)_d$, since $x^{d-1}y \in J_d$.
\end{example}

The following is a straightforward generalization of \cite[Theorem~4.3.4]{KKR2005}, \cite[Proposition~6.4.17]{kreuzer2005computational}.
\begin{proposition}[\textbf{Existence and Uniqueness}]\label{prop: eu basis}
	Let $\Oid$ be an infinite order ideal, let $J \subseteq P$ be a homogeneous ideal, and assume that, for every $d \geq 0$, the residue classes of the elements of $\Oid_d$ form a $\K$-vector space basis of $P_d/J_d$.
	\begin{enumerate}[label=(\roman*)]
		\item \label{it:eu basis1} There exists a unique homogeneous $\dO{1}$-basis of $J$.
		\item \label{it:eu basis2}Let $G$ be a homogeneous $\dO{1}$-prebasis whose elements are in $J$. Then $G$ is the homogeneous $\dO{1}$-basis of $J$.
	\end{enumerate}
\end{proposition}
\begin{proof}
	\begin{enumerate}[label=(\roman*), align=left, labelsep=0.5em, leftmargin=*]
		\item For $d\geq 0$, consider $\sigma \in \dO{1}_d$; by hypothesis, the residue class of $\sigma$ in $P_d/J_d$ is linearly dependent on the residue classes of the elements of $\Oid_d$, that is 
		\[
		\overline{\sigma}=\sum_{\tau \in \Oid_d}c_{\sigma\tau}\overline{\tau}
		\]
		with $c_{\sigma\tau} \in \K$. Therefore $J_d$ contains the polynomial
		\[
		g_\sigma = \sigma - \sum_{\tau \in \Oid_d}c_{\sigma\tau}\tau.
		\] 
		Observe that $\supp{g_\sigma} \cap \dO{1} = \{\sigma\}$, and we set $\h{g_\sigma}=\sigma$. Then 
		\[
		G = \left\{ g_\sigma \mid \sigma \in \dO{1}_d,\, d \geq 0 \right\} \subseteq J
		\]
		is a homogeneous $\dO{1}$-prebasis, and hence a homogeneous $\dO{1}$-basis of $J$ by Definition~\ref{def: border basis}.
        
		Now, let $G' = \left\{ g'_\sigma \mid \sigma \in \dO{1} \right\}$ be another homogeneous $\dO{1}$-basis of $J$. 
        Observe that for every $\sigma\in \dO{1}$, $g_\sigma-g'_\sigma$ belongs to $\langle \Oid\rangle$, since $\supp{g_\sigma}\cap \dO{1}=\supp{g'_\sigma}\cap \dO{1}=\{\sigma\}$.
        Therefore, $g_\sigma-g'_\sigma \in J_{\deg(\sigma)} \cap \langle\Oid_{\deg(\sigma)}\rangle$, hence $g_\sigma=g'_\sigma$ by Definition~\ref{def: border basis}.\ref{it:def border basis3}, since $G$ is a homogeneous $\dO{1}$-border basis of $J$.
		\item By Definition~\ref{def: border basis}, it suffices to observe that the set $G$ is a homogeneous $\dO{1}$-basis of $J$ and to apply~\ref{it:eu basis1}.
		\qedhere
	\end{enumerate}
\end{proof}

One might wonder whether a given homogeneous positive-dimensional ideal $J$ possesses a homogeneous border basis at all. By Proposition~\ref{prop: eu basis}.\ref{it:eu basis1}, we can rephrase this question as follows: Given a homogeneous nonzero-dimensional ideal $J$, are there order ideals such that, for every $d \geq 0$, the residue classes of their elements of degree $d$ form a $\K$-vector space basis of $(P/J)_d$? The answer is yes, and the argument is straightforward using Gr\"obner bases.
 
We recall that, given a term ordering $\preceq$ on $\T$, the initial ideal $\inid{\preceq}{J}$ of $J$ is the monomial ideal generated by all the leading terms of the polynomials in $J$: $\inid{\preceq}{J} = \left( \lt{\preceq}{f} : f \in J \setminus \{0\} \right)$.
\par
The following proposition is showed in \cite[Proposition~4.3.6]{KKR2005}, whose proof also remains valid for positive-dimensional ideals.
\begin{proposition}\label{prop: extension groebner}
	Let $J \subseteq P$ be a homogeneous positive-dimensional ideal, let $\preceq$ be a 
    term ordering on $\T$, and let $\Oid_{\preceq,J}$ be the order ideal $\T \setminus \inid{\preceq}{J}$. Then there exists a unique homogeneous $\Oid_{\preceq,J}$-border basis of $J$.
\end{proposition}

However, the next example shows that not every homogeneous border basis of $J$ arises from a term ordering. In this sense, the theory of homogeneous border bases of positive-dimensional ideals generalizes the theory of their Gr\"obner bases.
\begin{example}\label{es: extension groebner}
	In the polinomial ring $P=\K[x,y]$, with $\mathrm{char}(\K)=0$, consider the ideal $J=(x^3+x^2y+y^3)$.
	Since a Gr\"obner basis of $J$ is $\{x^3+x^2y+y^3\}$, and since the leading term of the polynomial $x^3+x^2y+y^3$ is either $x^3$ or $y^3$, the ideal $J$ has two possible initial ideals: namely, the ideals $I^{(1)} = (x^3)$ and $I^{(2)} = (y^3)$. They respectively lead to the order ideals $\Oid^{(1)} = \T \setminus I^{(1)} = \T(y) \cup x\cdot\T(y) \cup x^2\cdot\T(y)$ and $\Oid^{(2)} = \T \setminus I^{(2)} = \T(x) \cup y\cdot\T(x) \cup y^2\cdot\T(x)$.
	We now consider another order ideal $\Oid$ such that $|\Oid_d| = \dim_\K(P/J)_d, \, \forall d \geq 0$, namely $\Oid = \T \setminus (x^2y) = \T(x) \cup \T(y) \cup x\cdot\T(y)$. 
	The following diagrams illustrate the three order ideals.
	
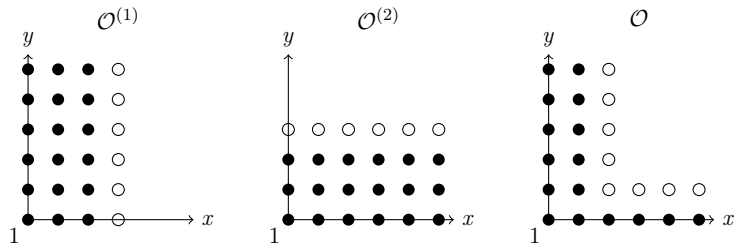
\begin{center}
		\begin{tikzpicture}
			\draw[->] (0,0) -- (5.5,0) node[right] {\textit{x}};
			\draw[->] (0,0) -- (0,5.5) node[above] {\textit{y}};
			\node[below left] at (0,0) {\text{1}};
			\foreach \x/\y in {0/0, 0/1, 0/2, 0/3, 0/4, 0/5, 1/0, 1/1, 1/2, 1/3, 1/4, 1/5, 2/0, 2/1, 2/2, 2/3, 2/4, 2/5} {
				\node[fill,circle,inner sep=2pt] at (\x,\y) {};
			}
			\foreach \x/\y in {3/0, 3/1, 3/2, 3/3, 3/4, 3/5} {
				\node[draw,circle,inner sep=2pt] at (\x,\y) {};
			}
			\node at (3,6.7) {$\Oid^{(1)}$};
		\end{tikzpicture}
		\quad
		\begin{tikzpicture}
			\draw[->] (0,0) -- (5.5,0) node[right] {\textit{x}};
			\draw[->] (0,0) -- (0,5.5) node[above] {\textit{y}};
			\node[below left] at (0,0) {\text{1}};
			\foreach \x/\y in {0/0, 1/0, 2/0, 3/0, 4/0, 5/0, 0/1, 1/1, 2/1, 3/1, 4/1, 5/1, 0/2, 1/2, 2/2, 3/2, 4/2, 5/2} {
				\node[fill,circle,inner sep=2pt] at (\x,\y) {};
			}
			\foreach \x/\y in {0/3, 1/3, 2/3, 3/3, 4/3, 5/3} {
				\node[draw,circle,inner sep=2pt] at (\x,\y) {};
			}
			\node at (3,6.7) {$\Oid^{(2)}$};
		\end{tikzpicture}
		\quad
		\begin{tikzpicture}
			\draw[->] (0,0) -- (5.5,0) node[right] {\textit{x}};
			\draw[->] (0,0) -- (0,5.5) node[above] {\textit{y}};
			\node[below left] at (0,0) {\text{1}};
			\foreach \x/\y in {0/0, 0/1, 0/2, 0/3, 0/4, 0/5, 1/0, 1/1, 1/2, 1/3, 1/4, 1/5, 2/0, 3/0, 4/0, 5/0} {
				\node[fill,circle,inner sep=2pt] at (\x,\y) {};
			}
			\foreach \x/\y in {2/1, 2/2, 2/3, 2/4, 2/5, 3/1, 4/1, 5/1} {
				\node[draw,circle,inner sep=2pt] at (\x,\y) {};
			}
			\node at (3,6.7) {$\Oid$};
		\end{tikzpicture}
	\end{center}
	Since $\Oid \neq \Oid^{(1)}$ and $\Oid \neq \Oid^{(2)}$, this order ideal does not correspond to an initial ideal arising from any term order. Yet, we have that the residue classes of the elements in $\Oid_d$ form a $\K$-vector space basis of $P_d/J_d, \, \forall d \geq 0$. 
	Indeed, we clearly have $P_d/J_d \cong  \langle \Oid_d \rangle$ for $d\leq 3$,
	and we aim to show that
	\[
	P_d/J_d \cong \langle x^d, xy^{d-1}, y^d \rangle = \langle \Oid_d \rangle \quad \forall d \geq 4.
	\]
	To do so, let us define the sequence in $\K$: $k_0=k_1=k_2 \coloneqq 1, \; k_i \coloneqq k_{i-1}+k_{i-3}$ for $i \geq 3$. \\
	First, we show by induction on $d \geq 4$ that:
	\begin{equation}\label{eq: es base to 1}
		\begin{aligned}
			x^d &
            = \left( \sum_{i=0}^{d-3} (-1)^{i}k_ix^{d-3-i}y^{i} \right)\cdot\left( x^3+x^2y+y^3 \right) + \\&\quad + 
            (-1)^{d-3}k_{d-4}xy^{d-1} + (-1)^{d-2}k_{d-2}x^2y^{d-2} + (-1)^{d-2}k_{d-3}y^d.
		\end{aligned}
	\end{equation}
	If $d=4$, the identity is verified by a direct calculation.
	By inductive hypothesis, suppose that (\ref{eq: es base to 1}) holds true, and let us prove it for $d+1$. Multiplying both sides of (\ref{eq: es base to 1}) by $x$, we have
	\begin{align*}
    x^{d+1}&
    = \left( \sum_{i=0}^{d-3} (-1)^{i}k_ix^{d-2-i}y^{i} \right)\cdot\left( x^3+x^2y+y^3 \right) + \\&\quad +
        (-1)^{d-3}k_{d-4}x^2y^{d-1} + (-1)^{d-2}k_{d-2}x^3y^{d-2} + (-1)^{d-2}k_{d-3}xy^d
	\end{align*}
	and by adding and subtracting the quantity $\left( (-1)^{d-2}k_{d-2}x^{0}y^{d-2} \right)\cdot\left( x^3+x^2y+y^3 \right)$ to the right-hand side of this equality, we obtain the desired expression also for $x^{d+1}$. 
    \\
	Now, let us fix $d \geq 4$. Equality~(\ref{eq: es base to 1}) leads to the equality in $P_d/J_d$
	\begin{equation}\label{eq: es base to 2}
		\overline{x^2y^{d-2}} = k_{d-2}^{-1} \left( (-1)^{d-2}\overline{x^d} +k_{d-4}\overline{xy^{d-1}} -k_{d-3}\overline{y^d} \right).
	\end{equation}
	Moreover,  for  $4 \leq i \leq d-1$, 
	\begin{equation}\label{eq: es base to 3}
		\begin{aligned}
			\overline{x^3y^{d-3}} &= \left( -\overline{x^2y} -\overline{y^3} \right)\overline{y^{d-3}} = -\overline{x^2y^{d-2}} -\overline{y^d}, \\
			\overline{x^iy^{d-i}} &= 
            (-1)^{i-3}k_{i-4}\overline{xy^{d-1}} + (-1)^{i-2}k_{i-2}\overline{x^2y^{d-2}} + (-1)^{i-2}k_{i-3}\overline{y^d}. \\
					\end{aligned}
	\end{equation}
	Then, $\left\{ \overline{x^d},\overline{xy^{d-1}},\overline{y^d} \right\}$ generate $P_d/J_d$, since we can substitute expression~(\ref{eq: es base to 2}) in equations~(\ref{eq: es base to 3}), and are linearly independent, since $\dim_\K(P/J)_d = 3 = |\Oid_d|$.
\end{example}

As a homogeneous positive-dimensional ideal $J$ always admits a homogeneous border basis $G$ (Proposition~\ref{prop: eu basis}), we will from now on identify $J$ with $(G)$.

We establish now a first characterization for a border prebasis $G$ to be a border basis  by the border reductors used in the rewriting process. This characterization adapts \cite[Proposition~14]{kehrein2005characterizations}, \cite[Proposition~6.4.28]{kreuzer2005computational} to our context, also taking into account \cite[Theorem~7.7]{ceria2019general}.

\begin{theorem}[\textbf{Characterization by Border Reductors}] \label{char_rewrite_relations} 
	Let $G$ be a homogeneous $\dO{1}$-prebasis,  and $\J$ a border reduction structure. The following statements are equivalent:
	\begin{enumerate}[label=(\roman*)]
		\item \label{it:char_rewrite1}$G$ is a homogeneous $\dO{1}$-basis;
		\item \label{it:char_rewrite2} for each $d \geq 0$, it holds $\langle \TT{G}_d \rangle = (G)_d$;
		\item \label{it:char_rewrite3}it holds $\langle \TT{G} \rangle = (G)$.
	\end{enumerate}
\end{theorem}
\begin{proof}
	\begin{enumerate}[label={}, align=left, labelsep=0.5em, leftmargin=*]
		\item[\ref{it:char_rewrite1}$\Rightarrow$\ref{it:char_rewrite2}:] Let us fix $d \geq 0$. It suffices to show that $(G)_d \subseteq \langle \TT{G}_d \rangle$. Let $f \in (G)_d \subseteq P_d$. We can uniquely write $f=q+h$, where $q \in \langle \TT{G}_d \rangle$ and $h \in \langle \Oid_d \rangle$ by Proposition~\ref{prop: noeth, direct sum, confl}. Then, in  $P_d/(G)_d$ we obtain $\overline{0} = \overline{f} = \overline{h}$, which means $h \in (G)_d$. Thus $h=0$ because $G$ is a basis by hypothesis, and so $f=q \in \langle \TT{G}_d \rangle$.
		\item[\ref{it:char_rewrite2}$\Rightarrow$\ref{it:char_rewrite1}:] By Proposition~\ref{prop: noeth, direct sum, confl}, for every $d\geq 0$ we have $P_d = \langle \TT{G}_d \rangle \oplus \langle \Oid_d \rangle$. By hypothesis, $\langle \TT{G}_d \rangle=(G)_d$ for every $d\geq 0$. Summing up these two facts, we obtain that $G$ fulfills Definition~\ref{def: border basis}.\ref{it:def border basis2}.
		\item[\ref{it:char_rewrite2}$\Leftrightarrow$\ref{it:char_rewrite3}:] This is immediate since $(G)$ is a homogeneous ideal. 
		\qedhere
	\end{enumerate}
\end{proof}

\section{Formal Multiplication Matrices}\label{sec matrices}

In this section we give a criterion for a border prebasis to be a border basis by  the commutativity of formal multiplication matrices. This characterization imitates the corresponding characterization for border bases \cite[Theorem~4.3.17]{KKR2005}, and the one in \cite{MourrainTreb}. However, in our homogenous setting we need to impose  the commutativity of two maps in consecutive degrees.
 
We consider the order ideal $\Oid = \left\{\tau_i\right\}_{i\in\N\setminus\{0\}}$ ordered according to ascending degree. For $d\in\N$, let $L(d)$ 
be the cardinality of $\Oid_{\leq d}$. Then $\Oid_{\leq d} = \left\{\tau_1,\dots,\tau_{L(d)}\right\}$ and $\Oid_d = \left\{\tau_{L(d-1)+1},\dots,\tau_{L(d)}\right\}$ (assuming $L(-1)=0$). With these notations, considering a homogeneous $\dO{1}$-prebasis $G$, the polynomials in $G_d$ can be written as 
\begin{equation}\label{eq:mpoly for multmatrix}
    g_\sigma = \sigma - \sum_{\ell=1}^{|\Oid_d|}c_{\sigma\tau_{L(d-1)+\ell}}\tau_{L(d-1)+\ell}.
\end{equation}
 
Since the residue classes of the elements of $\Oid_d$ generate $P_d/(G)_d$ as a $\K$-vector space (Corollary~\ref{cor prebasis generates}), we can describe the structure of this algebra by describing the effect of multiplying these generators by an indeterminate, generalizing \cite[Definition~15]{kehrein2005characterizations}. 

\begin{definition}\label{def: matrices}
	Let $G$ be a homogeneous $\dO{1}$-prebasis. Given $r \in \{0,\dots,n\}$ and $d\geq0$, we define the \emph{$r$-th $d$-graded formal multiplication matrix} $\Chi_r^{(d)} = (\xi_{k\,\ell}^{(r,d)}) \in \mat{|\Oid_{d+1}|,|\Oid_d|}{\K}$ of $G$ by
	\begin{equation}\label{eq:entries matrix}
	\xi_{k\,\ell}^{(r,d)} = 
	\begin{cases}
		\delta_{i\,k} & \text{if } x_r{\tau_{L(d-1)+\ell}} = {\tau_{L(d)+i}} \in \Oid \\
		c_{\sigma_j\tau_{L(d)+k}} & \text{if } x_r{\tau_{L(d-1)+\ell}} = {\sigma_j} \in \dO{1}
	\end{cases}
	\end{equation}
	where $\delta_{i\,k} = 1$ if $i = k$ and $\delta_{i\,k} = 0$ otherwise, and $c_{\sigma_j\tau_{L(d)+k}}$ is the coefficient of $\tau_{L(d)+k}$ in the polynomial $g_{\sigma_j}$  in the writing~\eqref{eq:mpoly for multmatrix}.
\end{definition}

Notice that all entries $\xi_{k\,\ell}^{(r,d)}$ of $\Chi_r^{(d)}$ are determined by the coefficients of the polynomials in $G_{d+1}$. 

\begin{remark}
Similarly to the interpretation of the formal multiplication matrices of a border prebasis in \cite[page~260]{kehrein2005characterizations}, the multiplication matrices of the homogeneous border prebasis $G$ can be costructed by the following procedure. Keeping in mind that the residue classes of the terms in $\Oid_d$ generate $P_{d}/(G)_d$, we multiply an
element of $\langle \Oid_d\rangle$ 
by the indeterminate $x_r$. Whenever $x_r{\tau_{L(d-1)+\ell}} = {\sigma_j}$ is a border term, we reduce it by the corresponding border polynomial $g_{\sigma_j}$ so that we represent it in $P_{d+1}/(G)_{d+1}$ by
 the residue classes of $\Oid_{d+1}$, that generate it. Hence, every element 
$v = c_1{\tau_{L(d-1)+1}} + \cdots + c_{|\Oid_d|}{\tau_{L(d)}} \in \langle\Oid_d\rangle$ is considered as a column vector $(c_1,\dots,c_{|\Oid_d|})^\mathrm{tr} \in \K^{|\Oid_d|}$, and $x_rv$ corresponds to $\Chi_r^{(d)}(c_1,\dots,c_{|\Oid_d|})^\mathrm{tr}$.
\end{remark}

\begin{example}\label{es: matrices 1}
	Consider the order ideal in Example~\ref{es: generators} ordered according to ascending degree in the following way:
	\begin{gather*}
		\Oid_0=\{{\tau_1}=1\}, \quad
        \Oid_1=\{{\tau_2}=x, \, {\tau_3}=y\}, \quad 
        \Oid_2=\{{\tau_4}=x^2, \, {\tau_5}=xy, \,{\tau_6}=y^2\}, \\ 
        \Oid_3=\{{\tau_7}=x^3, \, {\tau_8}=x^2y, \, {\tau_9}=xy^2, \, {\tau_{10}}=y^3\}, \quad
        \Oid_4=\{{\tau_{11}}=x^4, \, {\tau_{12}}=x^3y, \, {\tau_{13}}=x^2y^2\}, \\
        \Oid_d=\{{\tau_{2d+4}}=x^d, \, {\tau_{2d+5}}=x^{d-1}y\} \; \text{ for all } d\geq 5.
	\end{gather*}
	Let $G$ be the $\dO{1}$-prebasis of \eqref{eq:es border reductors}. In this case, for $d=3,4$  the formal multiplication matrices of $G$ are:
	{\allowdisplaybreaks\begin{gather*}
			\Chi^{(3)} = \begin{bmatrix}
            1	& 0	& 0	& 0 \\
            0	& 1	& 0	& 0 \\
            0	& 0	& 1	& 0 
	        \end{bmatrix}, \;
			\mathcal{Y}^{(3)} = \begin{bmatrix}
			0	& 0 & 0	& 0	\\
			1	& 0 & 0	& 0	\\
			0	& 1 & 0	& 0	
			\end{bmatrix}\; \in \mat{3,4}{\K}, \\
			\Chi^{(4)} = \begin{bmatrix}
			1	& 0	& -1 \\
			0	& 1	& 0 
			\end{bmatrix}, \;
			\mathcal{Y}^{(4)} = \begin{bmatrix}
			0	& -1	& -1 \\
			1	& 0		& 0 
			\end{bmatrix}\; \in \mat{2,3}{\K},
	\end{gather*}}
	where $\Chi^{(d)}$ (resp. $\mathcal{Y}^{(d)}$) is the formal multiplication matrix with respect to the variable $x$ (resp. $y$).
\end{example}

\begin{lemma}\label{lemma char comm mat}
Let $G$ be a homogeneous $\dO{1}$-prebasis, and $d_0\geq 2$. The following statements are equivalent:
	\begin{enumerate}[label=(\roman*)]
		\item\label{it_i_char comm mat} for every $0\leq d\leq d_0$, the residue classes of the elements of $\Oid_d$ are a $\K$-vector space basis of $P_d/(G)_d$;
		\item \label{it_ii_char comm mat} for every $0\leq d\leq d_0-2$, the graded formal multiplication matrices of $G$ satisfy the equality
		\[
		\Chi_r^{(d+1)} \Chi_s^{(d)} = \Chi_s^{(d+1)} \Chi_r^{(d)} \; \text{ for all } r,s \in \{0,\dots,n\}.
        \]
	\end{enumerate}
	In that case, for $0\leq d\leq d_0-1$, the $d$-graded formal multiplication matrices represent the multiplication homomorphisms $P_d/(G)_d \to P_{d+1}/(G)_{d+1}$ with respect to the bases $\{\overline{\tau} \mid \tau \in \Oid_d\}$ and $\{\overline{\tau} \mid \tau \in \Oid_{d+1}\}$.
\end{lemma}
\begin{proof}
	\begin{enumerate}[label={}, align=left, labelsep=0.5em, leftmargin=*]
		\item[\ref{it_i_char comm mat}$\Rightarrow$\ref{it_ii_char comm mat}:] Let us fix $0\leq d \leq d_0-2$. 
        Let $\Chi_0^{(d)},\dots,\Chi_n^{(d)}$ be the $d$-graded formal multiplication matrices of $G$, and let $\Chi_0^{(d+1)},\dots,\Chi_n^{(d+1)}$ be the $(d+1)$-graded ones. 
        By hypothesis, 
        the set $\{\overline{\tau} \mid \tau \in \Oid_d\} = \{\overline{{\tau_{L(d-1)+1}}},\dots,\overline{{\tau_{L(d)}}}\}$ is a $\K$-vector space basis of $P_d/(G)_d$, the set $\{\overline{\tau} \mid \tau \in \Oid_{d+1}\} = \{\overline{{\tau_{L(d)+1}}},\dots,\overline{\tau_{L(d+1)}}\}$ is a $\K$-vector space basis of $P_{d+1}/(G)_{d+1}$, and the set $\{\overline{\tau} \mid \tau \in \Oid_{d+2}\} = \{\overline{{\tau_{L(d+1)+1}}},\dots,\overline{\tau_{L(d+2)}}\}$ is a $\K$-vector space basis of $P_{d+2}/(G)_{d+2}$. 
        Therefore each matrix $\Chi_r^{(d)}$ defines a $\K$-linear map $\chi_r^{(d)} \colon P_d/(G)_d \to P_{d+1}/(G)_{d+1}$, i.e.
		\begin{align*}
			\chi_r^{(d)}(\overline{{\tau_{L(d-1)+1}}}) &= \xi_{1\,1}^{(r,d)}\,\overline{{\tau_{L(d)+1}}} + \cdots + \xi_{|\Oid_{d+1}|\,1}^{(r,d)}\,\overline{{\tau_{L(d+1)}}} \\
			&\;\;\vdots \\
			\chi_r^{(d)}(\overline{{\tau_{L(d)}}}) &= \xi_{1\,|\Oid_d|}^{(r,d)}\,\overline{{\tau_{L(d)+1}}} + \cdots + \xi_{|\Oid_{d+1}|\,|\Oid_d|}^{(r,d)}\,\overline{{\tau_{L(d+1)}}},
		\end{align*}
        and analogously each matrix $\Chi_r^{(d+1)}$ defines a $\K$-linear map $\chi_r^{(d+1)} \colon P_{d+1}/(G)_{d+1} \to P_{d+2}/(G)_{d+2}$. \\
        According to the definition of the entries $\xi_{k\,\ell}^{(r,d)}$ in~\eqref{eq:entries matrix}, only two cases occur: the product $x_r{\tau_{L(d-1)+\ell}} \in \T_{d+1}$ equals  either some term ${\tau_{L(d)+i}}$ in the order ideal $\Oid$ or some term ${\sigma_j}$ in the border $\dO{1}$. In the former case we have 
		\begin{align*}
			\chi_r^{(d)}(\overline{{\tau_{L(d-1)+\ell}}}) &= 0\,\overline{{\tau_{L(d)+1}}} + \cdots + 0\,\overline{{\tau_{L(d)+i-1}}} + 1\,\overline{{\tau_{L(d)+i}}} + 
            0\,\overline{{\tau_{L(d)+i+1}}} + \cdots + 0\,\overline{{\tau_{L(d+1)}}} = \\
			&= \overline{{\tau_{L(d)+i}}} = \overline{x_r{\tau_{L(d-1)+\ell}}},
		\end{align*}
		and in the latter case we have
        \[
        \chi_r^{(d)}(\overline{{\tau_{L(d-1)+\ell}}})=
        c_{\sigma_j\tau_{L(d)+1}}\overline{{\tau_{L(d)+1}}} + \cdots + c_{\sigma_j\tau_{L(d+1)}}\overline{{\tau_{L(d+1)}}} = 
        \overline{{\sigma_j}} = \overline{x_r{\tau_{L(d-1)+\ell}}}.
        \]
		From this it follows that the homomorphism $\chi_r^{(d)} \colon P_d/(G)_d \to P_{d+1}/(G)_{d+1}$ is multiplication by $x_r$, and similarly, the homomorphism $\chi_r^{(d+1)} \colon P_{d+1}/(G)_{d+1} \to P_{d+2}/(G)_{d+2}$ is multiplication by $x_r$. 
        
		Now, let $r,s \in \{0,\dots,n\}$. The matrices $\Chi_r^{(d+1)}\Chi_s^{(d)}$ and $\Chi_s^{(d+1)}\Chi_r^{(d)}$ represent the map compositions $\chi_r^{(d+1)}\circ\chi_s^{(d)}$ and $\chi_s^{(d+1)}\circ\chi_r^{(d)}$, respectively. For any basis element $\overline{{\tau_{L(d-1)+\ell}}} \in P_d/(G)_d$ we have
		\begin{align*}
			\chi_r^{(d+1)}\big(\chi_s^{(d)}(\overline{{\tau_{L(d-1)+\ell}}})\big) &= \chi_r^{(d+1)}(\overline{x_s{\tau_{L(d-1)+\ell}}}) = \chi_r^{(d+1)}\left(\sum_{k=1}^{|\Oid_{d+1}|} \xi_{k\,\ell}^{(s,d)}\overline{{\tau_{L(d)+k}}}\,\right) = \\
			&= \sum_{k=1}^{|\Oid_{d+1}|}\xi_{k\,\ell}^{(s,d)}\chi_r^{(d+1)}(\overline{{\tau_{L(d)+k}}}) = \sum_{k=1}^{|\Oid_{d+1}|}\xi_{k\,\ell}^{(s,d)}\overline{x_r{\tau_{L(d)+k}}} = \\
			&= \overline{x_rx_s{\tau_{L(d-1)+\ell}}},
		\end{align*}
		and analogously
		\[
		\chi_s^{(d+1)}\big(\chi_r^{(d)}(\overline{{\tau_{L(d-1)+\ell}}})\big) = \overline{x_sx_r{\tau_{L(d-1)+\ell}}}.
		\]
		Since the indeterminates commute, i.e. $x_rx_s=x_sx_r$, the two results coincide. Hence, by $\K$-linearity, $\chi_r^{(d+1)}\circ\chi_s^{(d)} = \chi_s^{(d+1)}\circ\chi_r^{(d)}$ and therefore $\Chi_r^{(d+1)}\Chi_s^{(d)} = \Chi_s^{(d+1)}\Chi_r^{(d)}$, for all $r,s \in \{0,\dots,n\}$.
		\item[\ref{it_ii_char comm mat}$\Rightarrow$\ref{it_i_char comm mat}:] For any $0\leq d\leq d_0$, we define a $\K$-linear map $\tilde{\Theta}^{(d)} \colon P_d \to \langle\Oid_d\rangle \colon f \mapsto \tilde{\Theta}^{(d)}(f)$ by setting
		\[
		\tilde{\Theta}^{(d)} (x_0^{\alpha_0} \cdots x_n^{\alpha_n}) = ({\tau_{L(d-1)+1}},\dots,{\tau_{L(d)}}) \prod_{k=0}^{n}\Chi_k^{(d-\sum_{i=0}^{k-1}\alpha_i-1)}\cdots\Chi_k^{(d-\sum_{i=0}^{k-1}\alpha_i-\alpha_k)}
		\]
		for each $x_0^{\alpha_0} \cdots x_n^{\alpha_n} \in \T_d$ when $d \geq 1$, and $\tilde{\Theta}^{(0)}(f) = f$. \\
		The map $\tilde{\Theta}^{(d)}$ is well defined. 
        Indeed, the matrix product $\prod_{k=0}^{n}\prod_{j=1}^{\alpha_k}\Chi_k^{(d-\sum_{i=0}^{k-1}\alpha_i-j)}$ lies in $\mat{|\Oid_d|,|\Oid_0|}{\K} = \mat{|\Oid_d|,1}{\K}$, and hence it can be multiplied with the row vector $({\tau_{L(d-1)+1}},\dots,{\tau_{L(d)}})$. 
        Moreover, the result does not depend on the ordering of the factors in the product $x_0^{\alpha_0} \cdots x_n^{\alpha_n}$, since by assumption $\Chi_r^{(d+1)} \Chi_s^{(d)} = \Chi_s^{(d+1)} \Chi_r^{(d)}$ for all $r,s \in \{0,\dots,n\}$ and all $0\leq d\leq d_0-2$. \\
		Furthermore, $\tilde{\Theta}^{(d)}$ is surjective: it suffices to show that the basis elements ${\tau_{L(d-1)+1}}, \dots, \allowbreak {\tau_{L(d)}}$ lie in its image. 
        To do so, we proceed by induction on the degree $d$ and prove that, for each $\tau \in \Oid_d$, it holds $\tilde{\Theta}^{(d)}(\tau) = \tau$. 
        Let $\mathcal{E}_\ell^{(d)}$ denote the matrix of size $|\Oid_d|\times1$ whose $\ell$-th entry is~$1$ and whose other entries are~$0$. 
        In the base case we have $\tilde{\Theta}^{(0)}({\tau_1}) = \tilde{\Theta}^{(0)}(1) = 1 = {\tau_1}$. 
        For the induction step, let ${\tau_{L(d-1)+i}} = x_r{\tau_{L(d-2)+\ell}} \in \Oid_{d}$ and assume by the inductive hypothesis that $\tilde{\Theta}^{(d-1)}({\tau_{L(d-2)+\ell}}) = {\tau_{L(d-2)+\ell}}$. 
        Then if ${\tau_{L(d-2)+\ell}} = x_0^{\alpha_0} \cdots x_n^{\alpha_n}$, by definition of $\tilde{\Theta}^{(d-1)}$ we have $\prod_{k=0}^{n}\prod_{j=1}^{\alpha_k}\Chi_k^{(d-1-\sum_{i=0}^{k-1}\alpha_i-j)} = \mathcal{E}_\ell^{(d-1)}$. Therefore we obtain
		\[
		\Chi_r^{(d-1)}\prod_{k=0}^{n}\prod_{j=1}^{\alpha_k}\Chi_k^{(d-1-\sum_{i=0}^{k-1}\alpha_i-j)} = \Chi_r^{(d-1)}\mathcal{E}_\ell^{(d-1)} = (\xi_{1\,\ell}^{(r,d-1)},\dots,\xi_{|\Oid_{d}|\,\ell}^{(r,d-1)})^\mathrm{tr} = \mathcal{E}_i^{(d)}
		\]
		by definition of $\Chi_r^{(d-1)}$ and ${\tau_{L(d-1)+i}}$, and thus
		\[
		\tilde{\Theta}^{(d)}({\tau_{L(d-1)+i}}) = ({\tau_{L(d-1)+1}},\dots,{\tau_{L(d)}})\,\Chi_r^{(d-1)}\prod_{k=0}^{n}\prod_{j=1}^{\alpha_k}\Chi_k^{(d-1-\sum_{i=0}^{k-1}\alpha_i-j)} = {\tau_{L(d-1)+i}}.
		\]
		
        From this we obtain an induced isomorphism of $\K$-vector spaces $\Theta^{(d)} \colon P_d/\ker\tilde{\Theta}^{(d)} \to \langle\Oid_d\rangle$. In particular, the residue classes ${\tau_{L(d-1)+1}}+\ker\tilde{\Theta}^{(d)}, \dots, {\tau_{L(d)}}+\ker\tilde{\Theta}^{(d)}$ are $\K$-linearly independent. \\
		Next we show that $(G)_d \subseteq \ker\tilde{\Theta}^{(d)}$. Let $\sigma = x_r{\tau_{L(d'-2)+s}} \in \dO{1}_{d'}$, with ${\tau_{L(d'-2)+s}} = x_0^{\alpha_0} \cdots x_n^{\alpha_n} \in \Oid_{d'-1}$, and $f = \sum_{m=1}^{u} c_mx_0^{\alpha'_{m0}} \cdots x_n^{\alpha'_{mn}} \in P_{d-d'}$ be such that $\deg(fg_\sigma) = d$. Define ${\tau_{L(d'-1)+\ell}} = x_0^{\alpha_{\ell 0}} \cdots x_n^{\alpha_{\ell n}}$ for every $\ell = 1,\dots,|\Oid_{d'}|$. Then we have
		\begin{align}\label{eq:proof comm char mat}
			\tilde{\Theta}^{(d)}(fg_\sigma) &= ({\tau_{L(d-1)+1}},\dots,{\tau_{L(d)}}) \sum_{m=1}^{u} c_m\prod_{k=0}^{n}\prod_{j=1}^{\alpha'_{mk}}\Chi_k^{(d-\sum_{i=0}^{k-1}\alpha'_{mi}-j)}\cdot \\ \nonumber
			&\quad \cdot\left( \Chi_r^{(d'-1)}\prod_{k=0}^{n}\prod_{j=1}^{\alpha_k}\Chi_k^{(d'-1-\sum_{i=0}^{k-1}\alpha_i-j)} 
            - \sum_{\ell=1}^{|\Oid_{d'}|}c_{\sigma\tau_{L(d'-1)+\ell}}\prod_{k=0}^{n}\prod_{j=1}^{\alpha_{\ell k}}\Chi_k^{(d'-\sum_{i=0}^{k-1}\alpha_{ik}-j)} \right)
		\end{align}
		and it is therefore sufficient to show that the part in the parentheses equals zero. As shown above, we have $\tilde{\Theta}^{(d'-1)}({\tau_{L(d'-2)+s}}) = {\tau_{L(d'-2)+s}}$ and thus $\prod_{k=0}^{n}\prod_{j=1}^{\alpha_k}\Chi_k^{(d'-1-\sum_{i=0}^{k-1}\alpha_i-j)} \allowbreak = \mathcal{E}_s^{(d'-1)}$, and $\tilde{\Theta}^{(d')}({\tau_{L(d'-1)+\ell}}) = {\tau_{L(d'-1)+\ell}}$ and thus $\prod_{k=0}^{n}\prod_{j=0}^{\alpha_{\ell k}}\Chi_k^{(d'-\sum_{i=0}^{k-1}\alpha_{ik}-j)} = \mathcal{E}_\ell^{(d')}$. It follows that
			\[\Chi_r^{(d'-1)}\prod_{k=0}^{n}\prod_{j=1}^{\alpha_k}\Chi_k^{(d'-1-\sum_{i=0}^{k-1}\alpha_i-j)} 
			= \Chi_r^{(d'-1)}\mathcal{E}_s^{(d'-1)} = (c_{\sigma\tau_{L(d'-1)+1}},\dots,c_{\sigma\tau_{L(d')}})^\mathrm{tr}\]
		and
			\[\sum_{\ell=1}^{|\Oid_{d'}|}c_{\sigma\tau_{L(d'-1)+\ell}}\prod_{k=0}^{n}\prod_{j=1}^{\alpha_{\ell k}}\Chi_k^{(d'-\sum_{i=0}^{k-1}\alpha_{ik}-j)}
			= \sum_{\ell=1}^{|\Oid_{d'}|}c_{\sigma\tau_{L(d'-1)+\ell}}\mathcal{E}_\ell^{(d')} = (c_{\sigma\tau_{L(d'-1)+1}},\dots,c_{\sigma\tau_{L(d')}})^\mathrm{tr}.\]
		Since these two expressions coincide, the part of \eqref{eq:proof comm char mat} in the parentheses vanishes. Therefore we obtain $fg_\sigma \in \ker\tilde{\Theta}^{(d)}$ for every $g_\sigma \in G, f \in P$ such that $\deg(fg_\sigma) = d$, and thus $(G)_d \subseteq \ker\tilde{\Theta}^{(d)}$, as desired. \\
		Hence, there is a natural surjective homomorphism of $\K$-vector spaces $\Psi \colon P_d/(G)_d \to P_d/\ker\tilde{\Theta}^{(d)}$. Since the set $\{{\tau_{L(d-1)+1}}+(G)_d, \dots, {\tau_{L(d)}}+(G)_d\}$ generates the $\K$-vector space $P_d/(G)_d$, and since the set $\{{\tau_{L(d-1)+1}}+\ker\tilde{\Theta}^{(d)}, \dots, {\tau_{L(d)}}+\ker\tilde{\Theta}^{(d)}\}$ is $\K$-linearly independent, both sets must be bases and $(G)_d = \ker\tilde{\Theta}^{(d)}$. 
        This shows that 
        the residue classes of the elements of $\Oid_d$ are a $\K$-vector space basis of $P_d/(G)_d$.
		\qedhere
	\end{enumerate}
\end{proof}

Thanks to Lemma~\ref{lemma char comm mat}, we obtain the following, which generalizes the well-known characterization for classical border bases \cite[Proposition~16]{kehrein2005characterizations}.
\begin{theorem}[\textbf{Characterization by Formal Multiplication Matrices}]\label{char comm mat}
	Let $G$ be a homogeneous $\dO{1}$-prebasis. The following statements are equivalent:
	\begin{enumerate}[label=(\roman*)]
		\item $G$ is a homogeneous $\dO{1}$-basis;
		\item \label{it_ii_char comm mat bis} for every $d\geq0$, the graded formal multiplication matrices of $G$ satisfy the equality
		\[
		\Chi_r^{(d+1)} \Chi_s^{(d)} = \Chi_s^{(d+1)} \Chi_r^{(d)}
		\]
		for all $r,s \in \{0,\dots,n\}$.
	\end{enumerate}
\end{theorem}

\begin{proof} By Theorem~\ref{char_rewrite_relations}, it is enough to prove that 
 item~\ref{it_ii_char comm mat bis} is equivalent to item~\ref{it:char_rewrite2} of Theorem~\ref{char_rewrite_relations}. By Lemma~\ref{lemma char comm mat}, this is immediate.
\end{proof}

\begin{example}\label{es: matrices 2}
	Recall Example~\ref{es: matrices 1}. The condition $\Chi^{(4)}\mathcal{Y}^{(3)}=\mathcal{Y}^{(4)}\Chi^{(3)}$ between the $3$- and $4$-graded formal multiplication matrices of $G$ does not hold:
	\[
	\Chi^{(4)}\mathcal{Y}^{(3)} = 
	\begin{bmatrix}
		0 & -1 & 0 & 0 \\
		1 & 0 & 0 & 0 
	\end{bmatrix}
	\neq 
	\begin{bmatrix}
		0 & -1 & -1 & 0 \\
		1 & 0 & 0 & 0 
	\end{bmatrix}
	= \mathcal{Y}^{(4)}\Chi^{(3)}.
	\]
	Hence $G$ is not a border basis, as already established in Example~\ref{es: generators}.
\end{example}

In order to verify whether $G$ is a border basis, it is enough to check the conditions in Theorem~\ref{char comm mat} starting from a sufficiently large degree $d$. In fact, for small degrees the graded multiplication matrices are determined uniquely by the order ideal $\Oid$, and do not depend on the choice of $G$.

If $S$ is a set of homogeneous polynomials, we denote $\mindeg{S}$ the minimum of the degrees of the polynomials in $S$.
\begin{corollary}\label{cor: comm mat}
	Let $G$ be a homogeneous $\dO{1}$-prebasis. The following statements are equivalent:
	\begin{enumerate}[label=(\roman*)]
		\item $G$ is a homogeneous $\dO{1}$-basis;
		\item for every $d\geq\mindeg{\dO{1}}-1$, the graded formal multiplication matrices of $G$ satisfy 
        \[\Chi_r^{(d+1)} \Chi_s^{(d)} = \Chi_s^{(d+1)} \Chi_r^{(d)} \; \text{ for all } 0 \leq r < s \leq n.
        \]
	\end{enumerate}
\end{corollary}
\begin{proof}
	It is enough to prove that for degrees $d < \mindeg{\dO{1}} - 1$ the condition $\Chi_r^{(d+1)}\Chi_s^{(d)}\mathcal{E}_\ell^{(d)} = \Chi_s^{(d+1)}\Chi_r^{(d)}\mathcal{E}_\ell^{(d)}$ is automatically satisfied for every $\ell = 1,\dots,|\Oid_d|$. \\
	Let us fix $0 \leq d < \mindeg{\dO{1}} - 1$, and fix ${\tau_{L(d-1)+\ell}} \in \Oid_d$. Two cases can occur:
	\begin{enumerate}[label={}, align=left, labelsep=0.5em, leftmargin=*]
		\item[Case 1:] 
		\begin{tabular}{@{}c l@{}}
			\renewcommand{\arraystretch}{1.1} 
			\begin{tabular}{|c|c|}
				\hline
				${\tau_{L(d)+j}}$ & ${\tau_{L(d+1)+k}}$ \\ \hline
				${\tau_{L(d-1)+\ell}}$ & ${\tau_{L(d)+i}}$ \\ \hline
			\end{tabular}
			&
			\parbox{0.65\textwidth}{
				$x_rx_s{\tau_{L(d-1)+\ell}} = {\tau_{L(d+1)+k}} \in \Oid$ and \\
                $x_r{\tau_{L(d-1)+\ell}} = {\tau_{L(d)+i}} \in \Oid, \, x_s{\tau_{L(d-1)+\ell}} = {\tau_{L(d)+j}} \in \Oid$. }
		\end{tabular} \\
		Then we have
        \[\Chi_r^{(d+1)}\Chi_s^{(d)}\mathcal{E}_\ell^{(d)} = \Chi_r^{(d+1)}\mathcal{E}_j^{(d+1)} = \mathcal{E}_k^{(d+2)} = \Chi_s^{(d+1)}\mathcal{E}_i^{(d+1)} = \Chi_s^{(d+1)}\Chi_r^{(d)}\mathcal{E}_\ell^{(d)},\]
        i.e. the desired condition holds by the definition of the formal multiplication matrices.
		\item[Case 2:] 
		\begin{tabular}{@{}c l@{}}
			\renewcommand{\arraystretch}{1.1} 
			\begin{tabular}{|c|c|}
				\hline
				${\tau_{L(d)+j}}$ & ${\sigma_k}$ \\ \hline
				${\tau_{L(d-1)+\ell}}$ & ${\tau_{L(d)+i}}$ \\ \hline
			\end{tabular}
			&
			\parbox{0.65\textwidth}{
				$x_rx_s{\tau_{L(d-1)+\ell}} = {\sigma_k} \in \dO{1}$ and \\ $x_r{\tau_{L(d-1)+\ell}} = {\tau_{L(d)+i}} \in \Oid, \, x_s{\tau_{L(d-1)+\ell}} = {\tau_{L(d)+j}} \in \Oid$. }
		\end{tabular} \\
		Note that in this case $d = \mindeg{\dO{1}} - 2$. 
		We have 
        \[\Chi_r^{(d+1)}\Chi_s^{(d)}\mathcal{E}_\ell^{(d)} = \Chi_r^{(d+1)}\mathcal{E}_j^{(d+1)} = (c_{\sigma_k\tau_{L(d+1)+1}},\dots,c_{\sigma_k\tau_{L(d+2)}})^\mathrm{tr} = \Chi_s^{(d+1)}\mathcal{E}_i^{(d+1)}  = \Chi_s^{(d+1)}\Chi_r^{(d)}\mathcal{E}_\ell^{(d)}.\] Again, the desired condition follows immediately from the definition of the formal multiplication matrices.
	\end{enumerate}
	We conclude that the condition $\Chi_r^{(d+1)}\Chi_s^{(d)} = \Chi_s^{(d+1)}\Chi_r^{(d)}$ is always satisfied for any $d < \mindeg{\dO{1}}-1$.
\end{proof}

We now present an example in which a prebasis is indeed a border basis.
\begin{example}\label{es: basis}
	Consider the order ideal $\Oid$ of Example~\ref{es: generators} ordered as in Example~\ref{es: matrices 1}, and consider another homogeneous $\dO{1}$-prebasis $G'$:
	\begin{gather*}
		g'_{y^4} = y^4 + x^3y, \quad
		g'_{xy^3} = xy^3 + x^4, \quad
		g'_{x^2y^3} = x^2y^3 + x^5, \\
		g'_{x^3y^2} = x^3y^2 + x^5 - x^4y, \quad
		g'_{x^ky^2} = x^ky^2 - x^{k+1}y + x^{k+2} \; \text{ for } k \geq 4.
	\end{gather*}
	Then, with the polynomials above, for $d\geq\mindeg{\dO{1}}-1 = 3$, we obtain the following $d$-graded formal multiplication matrices:
	{\allowdisplaybreaks\begin{gather*}
			\Chi^{(3)} = \begin{bmatrix}
				1 & 0 & 0 & -1 \\
				0 & 1 & 0 & 0 \\
				0 & 0 & 1 & 0 
			\end{bmatrix}, \;
			\mathcal{Y}^{(3)} = \begin{bmatrix}
				0 & 0 & -1 & 0 \\
				1 & 0 & 0 & -1 \\
				0 & 1 & 0 & 0 
			\end{bmatrix}\; \in \mat{3,4}{\K}; \\
			\Chi^{(4)} = \begin{bmatrix}
				1 & 0 & -1 \\
				0 & 1 & 1 
			\end{bmatrix}, \;
			\mathcal{Y}^{(4)} = \begin{bmatrix}
				0 & -1 & -1 \\
				1 & 1 & 0 
			\end{bmatrix}\; \in \mat{2,3}{\K}; \\
			\Chi^{(\ell)} = \begin{bmatrix}
				1 & 0 \\
				0 & 1 
			\end{bmatrix}, \;
			\mathcal{Y}^{(\ell)} = \begin{bmatrix}
				0 & -1 \\
				1 & 1 
			\end{bmatrix}\; \in \mat{2,2}{\K}  \quad \text{ for } \ell \geq 5.
	\end{gather*}}
	In this case, we have
	{\allowdisplaybreaks\begin{gather*}
			\Chi^{(4)}\mathcal{Y}^{(3)} = \begin{bmatrix}
				0 & -1 & -1 & 0 \\
				1 & 1 & 0 & -1 
			\end{bmatrix} = \mathcal{Y}^{(4)}\Chi^{(3)}, \\
			\Chi^{(5)}\mathcal{Y}^{(4)} = \begin{bmatrix}
				0 & -1 & -1 \\
				1 & 1 & 0 
			\end{bmatrix} = \mathcal{Y}^{(5)}\Chi^{(4)},
	\end{gather*}}
	and, for every $\ell \geq 5$, the relation $\Chi^{(\ell+1)}\mathcal{Y}^{(\ell)} = \mathcal{Y}^{(\ell+1)}\Chi^{(\ell)}$ is always satisfied. 
	Therefore, by Corollary~\ref{cor: comm mat}, $G'$ is a homogeneous $\dO{1}$-basis.
\end{example}

\section{Finite Number of Conditions for Border Bases}\label{sec finite conditions}

The characterization of border bases given in Theorem~\ref{char comm mat} lacks of effectiveness, since item~\ref{it_ii_char comm mat bis} has to be checked for all $d\geq 0$.

In Example~\ref{es: basis}, for $d>4$ the formal matrices always commute, since one of them is the identity matrix. In general, this is not the case. Indeed, if we consider ideals with Krull dimension $>1$, for $d\gg 0$ the formal matrices we construct are not square.

\begin{example}
In the polynomial ring $P=\K[x,y,z]$, with $\mathrm{char}(\K)=0$, consider the order ideal $\Oid=\mathbb T(x)\cup\mathbb{T}(y,z)$, whose border is $\dO{1} = (y\cdot\T(x)\setminus\{y\}) \cup (z\cdot\T(x)\setminus\{z\}) \cup (x\cdot\T(y,z)\setminus\{x\})$
. For every $d\geq 1$, the formal multiplication matrices have $\vert \Oid_{d+1}\vert = d+3$ rows and $\vert \Oid_d\vert = d+2$ columns, hence they are never square matrices.

\end{example}

Actually, since we deal with homogeneous ideals in the  polynomial ring $P$ which is Noetherian, it is quite simple to prove that it is sufficient to check item~\ref{it_ii_char comm mat bis} of Theorem~\ref{char comm mat} for a finite number of values~$d$.

We recall that for every homogeneous ideal $J\subseteq P$, the \emph{Hilbert function} of $P/J$ is
$H_{P/J}\colon \mathbb{Z} \to \mathbb{Z}$ such that $H_{P/J}(t)$ is the dimension of $P_t/J_t$ as a $\K$-vector space. 

\begin{definition}\cite[Chapter~4, page~55]{GeomSyz}
The \emph{(Castelnuovo-Mumford) regularity} of a homogeneous ideal
$J\subseteq P$, denoted $\reg{J}$, is the minimum $m$ such that $J$ is generated in degree $\leq m$, and the $k$-th module of syzygies of $J$ is generated in degree $\leq m+k$ for $k=1,\dots,n+1$. We say that $J$ is \emph{$d$-regular} if $\reg{J}\leq d$.
\end{definition}

\begin{definition}\cite[Definition~3.2]{Green}
Given $a, d \in \N$, the \emph{$d$-th Macaulay representation} of $a$ is the unique expression
\[a ={k_d\choose d}+{k_{d-1}\choose d-1}+\cdots  +{k_\delta\choose \delta},\]
with $\delta \in \Z$, satisfying $k_d >\cdots> k_\delta \geq \delta> 0$ (by convention ${a \choose b}=0$ if $a<b$). Given this representation, the \emph{$d$-th Macaulay transformation} of $a$ is 
\[a^{\langle d \rangle}={k_d + 1\choose d + 1}+{k_{d-1} + 1\choose d}+\cdots+{k_\delta + 1\choose \delta + 1}.\]
\end{definition}

The notion of regularity of an ideal $J$ and the Macaulay representation and transformation of $H_{P/J}(t)$, for some suitable $t$, are related by the following theorem. We give its statement as in \cite[Theorem 3.10]{MourrainTreb} and \cite[Corollary 4.2.14, Theorem 4.3.3]{BH}, but see also \cite{Gotz}, \cite[Theorems~3.8 and~3.11]{Green} for the classical statements concerning saturated ideals. 
\begin{theorem}\label{thms gotz}
Let $J$ be a homogeneous ideal in $P$. Set $h_i=H_{P/J}(i)$ for $i\geq 0$
. Let $t$ be such that $h_{t+1}=h_t^{\langle t \rangle}$ and assume that $J$ is generated in degrees $\leq t+1$. Then:
\begin{enumerate}[label=(\roman*)]
\item \label{thm Gotz pers} (Gotzmann’s Persistence Theorem) $h_{i+1}=h_i^{\langle i \rangle}$ for every $i\geq t$;
\item  \label{thm Gotz reg}(Gotzmann's Regularity Theorem) $\reg{J}\leq t$.
\end{enumerate}
\end{theorem}

We can now give an effective criterion to check whether a border prebasis is a basis.

We remind that $H_{P/(\dO{1})}(d)=\vert \Oid_d\vert$, $\forall d\geq 0$. 
\begin{theorem}\label{effective char comm mat}
	Let $\Oid$ be an order ideal, 
    and $G$ be a homogeneous $\dO{1}$-prebasis. Set $h_i=H_{P/(\dO{1})}(i)$ for $i\geq 0$.
    Let $t$ be such that $h_{t+1}=h_t^{\langle t \rangle}$ and
    assume that $(\dO{1})$ and $(G)$ are both generated in degrees $\leq t+1$.
    The following statements are equivalent:
	\begin{enumerate}[label=(\roman*)]
		\item\label{it:effective char comm mat1} $G$ is a homogeneous $\dO{1}$-basis;
        \item \label{it:effective char comm mat3} 
        for every $0\leq d\leq 
        t-1$, the graded formal multiplication matrices of $G$ satisfy the equality
		\[
		\Chi_r^{(d+1)} \Chi_s^{(d)} = \Chi_s^{(d+1)} \Chi_r^{(d)} \; \text{ for all } r,s \in \{0,\dots,n\}.
		\]
		\end{enumerate}
\end{theorem}
\begin{proof}
    \begin{enumerate}[label={}, align=left, labelsep=0.5em, leftmargin=*]
        \item[\ref{it:effective char comm mat1}$\Rightarrow$\ref{it:effective char comm mat3}:] Since $G$ is a homogeneous $\dO{1}$-basis, Theorem~\ref{char comm mat}.\ref{it_ii_char comm mat} holds for every $d\geq0$. 

		\item[\ref{it:effective char comm mat3}$\Rightarrow$\ref{it:effective char comm mat1}:] Thanks to the hypothesis, for every $0\leq d\leq t+1$, the residue classes of the elements in $\Oid_d$ are a $\K$-vector space basis of $P_d/(G)_d$ by Lemma \ref{lemma char comm mat}. \\
        As a consequence, we have that $H_{P/(G)}(d)=H_{P/(\dO{1})}(d)=\vert \Oid_d\vert$ for every $0\leq d\leq t+1$. In particular, by Theorem~\ref{thms gotz}.\ref{thm Gotz reg} $\reg{(\dO{1})}$ and $\reg{(G)}$ are both $\leq t$; furthermore, these ideals are both generated in degrees~$\leq t$. 
        By Theorem~\ref{thms gotz}.\ref{thm Gotz pers},
        we have that  $H_{P/(G)}(d)=\vert \Oid_d\vert$ for every~$d\geq t$. \\        
        Since the residue classes of $\Oid_d$ are a system of generators of $P_d/(G)_d$ for every $d\geq 0$, being $G$ a homogeneous $\dO{1}$-prebasis (Corollary~\ref{cor prebasis generates}), we obtain that the residue classes of $\Oid_d$ are a basis of $P_d/(G)_d$ for every $d\geq 0$.  
        \end{enumerate}
\end{proof}

Exploiting Corollary~\ref{cor: comm mat} and Lemma~\ref{lemma char comm mat}, we can refine Theorem~\ref{effective char comm mat} in the following way.
\begin{corollary}
    Let $\Oid$ be an order ideal,  and $G$ be a homogeneous $\dO{1}$-prebasis. Set $h_i=H_{P/(\dO{1})}(i)$ for $i\geq 0$.
    Let $t$ be such that $h_{t+1}=h_t^{\langle t \rangle}$ and
    assume that $(\dO{1})$ and $(G)$ are both generated in degrees $\leq t+1$.
    The following statements are equivalent:
    \begin{enumerate}[label=(\roman*)]
		\item\label{it:effective char border red1} $G$ is a homogeneous $\dO{1}$-basis;
        \item  
        for every $\mindeg{\dO{1}}-1\leq d\leq t-1$, the graded formal multiplication matrices of $G$ satisfy the equality
		\[
		\Chi_r^{(d+1)} \Chi_s^{(d)} = \Chi_s^{(d+1)} \Chi_r^{(d)} \; \text{ for all } 0\leq r<s\leq n;
		\]
		  \item 
        for every $\mindeg{\dO{1}}+1\leq d\leq t+1$, the residue classes of $\Oid_d$ are a $\K$-vector space basis of $P_d/(G)_d$.
    \end{enumerate}
\end{corollary}

We illustrate the application of Theorem~\ref{effective char comm mat} with a final example.

\begin{example}\label{ex finale}
    In the polynomial ring $K[x,y,z]$, consider the order ideal $\Oid=\T(x)\cup \T(y,z)$. The border of $\Oid$ is
    \[
   \dO{1}_0=\dO{1}_1=\emptyset, \quad \dO{1}_\ell=\{x^{\ell-1}y, x^{\ell-1}z\}\cup x\cdot\T(y,z)_{\ell-1}\; \forall \ell\geq 2.
    \]
    The following diagram illustrates the situation. The point $(r,s,t)$ of the diagram corresponds to the term $x^ry^sz^t$.
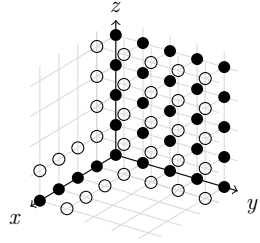
    \begin{center}
		\tdplotsetmaincoords{65}{125} 
		\begin{tikzpicture}[tdplot_main_coords, scale=1.1]
			\foreach \x in {0,...,4} {
				\draw[gray!30] (\x,0,0) -- (\x,4.5,0); 
			}
			\foreach \y in {0,...,4} {
				\draw[gray!30] (0,\y,0) -- (4.5,\y,0); 
			}
			\foreach \y in {0,...,4} {
				\draw[gray!30] (0,\y,0) -- (0,\y,4.5); 
			}
			\foreach \z in {0,...,4} {
				\draw[gray!30] (0,0,\z) -- (0,4.5,\z); 
			}
			\foreach \x in {0,...,4} {
				\draw[gray!30] (\x,0,0) -- (\x,0,4.5); 
			}
			\foreach \z in {0,...,4} {
				\draw[gray!30] (0,0,\z) -- (4.5,0,\z); 
			}
			\foreach \y in {0,...,4} {
				\draw[gray!30] (1,\y,0) -- (1,\y,4.5); 
			}
			\foreach \z in {0,...,4} {
				\draw[gray!30] (1,0,\z) -- (1,4.5,\z); 
			}
			\foreach \y in {0,...,4} \foreach \z in {0,...,4} {
				\draw[gray!30] (0,\y,\z) -- (1,\y,\z);
			}

			\draw[->] (0,0,0) -- (4.5,0,0) node[anchor=north east]{\textit{x}};
			\draw[->] (0,0,0) -- (0,4.5,0) node[anchor=north west]{\textit{y}};
			\draw[->] (0,0,0) -- (0,0,4.5) node[anchor=south]{\textit{z}};
			\foreach \x/\y/\z in {
				0/0/0, 
				1/0/0, 0/0/1, 0/1/0, 
				2/0/0, 0/0/2, 0/2/0, 0/1/1, 
				3/0/0, 0/0/3, 0/3/0, 0/2/1, 0/1/2, 
				4/0/0, 0/0/4, 0/4/0, 0/3/1, 0/2/2, 0/1/3, 
				0/4/1, 0/3/2, 0/2/3, 0/1/4, 
				0/4/2, 0/3/3, 0/2/4,
				0/4/3, 0/3/4, 0/4/4} {
				\node[fill,circle,inner sep=2pt] at (\x,\y,\z) {};
			}
			\foreach \x/\y/\z in {
				1/1/0, 1/0/1, 
				2/1/0, 2/0/1, 1/2/0, 1/0/2, 1/1/1, 
				3/1/0, 3/0/1, 1/3/0, 1/0/3, 1/1/2, 1/2/1, 
				4/1/0, 4/0/1, 1/4/0, 1/0/4, 1/1/3, 1/2/2, 1/3/1,  
				1/1/4, 1/2/3, 1/3/2, 1/4/1,
				1/2/4, 1/3/3, 1/4/2,
				1/3/4, 1/4/3, 1/4/4} {
				\node[draw,circle,inner sep=2pt] at (\x,\y,\z) {};
			}
		\end{tikzpicture}
	\end{center}    
    We have $(\dO{1})=(xy,xz)$. Let $G$ be a $\dO{1}$-prebasis. We consider the coefficients of the polynomials in $G$ as parameters, and we impose that $G$ is a $\dO{1}$-basis by imposing that the coefficients satisfy some polynomial equations.  By Theorems~\ref{thms gotz} and~\ref{char comm mat}, we only need to construct the formal multiplication matrices for $d=1,2$ (using also Corollary~\ref{cor: comm mat}). These multiplication matrices involve the coefficients of the polynomials in $G_2$ and $G_3$, i.e. 
    \[
    G_2=\{xy+c_{2,1,1}x^{2} +c_{2,1,2}y^{2} +c_{2,1,3}y z +c_{2,1,4}z^{2} ,\, xz+c_{2,2,1}x^{2} +c_{2,2,2}y^{2} +c_{2,2,3}y z +c_{2,2,4}z^{2}\}  
    \]
    {\allowdisplaybreaks\begin{equation*}
    \begin{split}
        G_3=\{&x^2 y+c_{3,1,1}x^{3} +c_{3,1,2}y^{3} +c_{3,1,3}y^2 z +c_{3,1,4}y z^2 +c_{3,1,5}z^{3},\\
        &x^{2} z + c_{3,2,1}x^{3} +c_{3,2,2}y^{3} +c_{3,2,3}y^{2} z +c_{3,2,4}y z^{2} +c_{3,2,5}z^{3} ,\\
        &xy^{2} +c_{3,3,1} x^{3} +c_{3,3,2}y^{3} +c_{3,3,3}y^{2} z +c_{3,3,4}y z^{2} +c_{3,3,5}z^{3} , \\
        & x z^{2}+c_{3,4,1} x^{3} +c_{3,4,2}y^{3} +c_{3,4,3}y^{2} z +c_{3,4,4}y z^{2} +c_{3,4,5}z^{3} ,\\
        &xy z+ c_{3,5,1}x^{3} +c_{3,5,2}y^{3} +c_{3,5,3}y^{2} z +c_{3,5,4}y z^{2} +c_{3,5,5} z^{3} \}
    \end{split}
    \end{equation*}}
    where $c_{i,j,k}\in \K$. We denote by $\Chi^{(d)}$, $\mathcal{Y}^{(d)}$, $\mathcal{Z}^{(d)}$ the formal multiplication matrices w.r.t. the variables $x$, $y$, $z$. For $d=1,2$, we have: 
    \[
    \Chi^{(1)}=\left[\begin{array}{ccc}
    1 & -c_{2,1,1} & -c_{2,2,1} 
    \\
    0 & -c_{2,1,2} & -c_{2,2,2} 
    \\
    0 & -c_{2,1,3} & -c_{2,2,3} 
    \\
    0 & -c_{2,1,4} & -c_{2,2,4} 
    \end{array}\right],\; 
    \Chi^{(2)}=\left[\begin{array}{cccc}
    1 & -c_{3,3,1} & -c_{3,5,1} & -c_{3,4,1} 
    \\
    0 & -c_{3,3,2} & -c_{3,5,2} & -c_{3,4,2} 
    \\
    0 & -c_{3,3,3} & -c_{3,5,3} & -c_{3,4,3} 
    \\
    0 & -c_{3,3,4} & -c_{3,5,4} & -c_{3,4,4} 
    \\
    0 & -c_{3,3,5} & -c_{3,5,5} & -c_{3,4,5} 
    \end{array}\right], \]
    \[    
    \mathcal{Y}^{(1)}=\left[\begin{array}{ccc}
    -c_{2,1,1} & 0 & 0 
    \\
    -c_{2,1,2} & 1 & 0 
    \\
    -c_{2,1,3} & 0 & 1 
    \\
    -c_{2,1,4} & 0 & 0 
    \end{array}\right], \; 
    \mathcal Y^{(2)}=
    \left[\begin{array}{cccc}
    -c_{3,1,1} & 0 & 0 & 0 
    \\
    -c_{3,1,2} & 1 & 0 & 0 
    \\
    -c_{3,1,3} & 0 & 1 & 0 
    \\
    -c_{3,1,4} & 0 & 0 & 1 
    \\
    -c_{3,1,5} & 0 & 0 & 0 
    \end{array}\right],\]
    \[
    \mathcal Z^{(1)}=\left[\begin{array}{ccc}
    -c_{2,2,1} & 0 & 0 
    \\
    -c_{2,2,2} & 0 & 0 
    \\
    -c_{2,2,3} & 1 & 0 
    \\
    -c_{2,2,4} & 0 & 1 
    \end{array}\right],
    \; \mathcal Z^{(2)}=\left[\begin{array}{cccc}
    -c_{3,2,1} & 0 & 0 & 0 
    \\
     -c_{3,2,2} & 0 & 0 & 0 
    \\
    -c_{3,2,3} & 1 & 0 & 0 
    \\
    -c_{3,2,4} & 0 & 1 & 0 
    \\
    -c_{3,2,5} & 0 & 0 & 1 
    \end{array}\right],
    \]
    where the matrices are written  considering the terms of $\Oid_d$
    , for $d=1,2,3 $, ordered in the following way:
    \begin{equation*}
    \begin{gathered}
    \Oid_0=\{ \tau_1=1\}, \quad
    \Oid_1=\{ \tau_2=x,\tau_3=y,\tau_4=z\}, \quad 
    \Oid_2=\{\tau_5=x^2, \tau_6=y^2, \tau_7=yz, \tau_8=z^2\},\\ 
    \Oid_3=\{\tau_9=x^{3},\tau_{10}=y^{3}, \tau_{11}=y^{2} z,\tau_{12}=y z^{2},\tau_{13}= z^{3}\}.
    \end{gathered}
    \end{equation*}
    By imposing the commutativity conditions between the products of formal multiplication matrices of Theorem~\ref{effective char comm mat}, and imposing that $(G)$ is generated in degree $\leq 3$, we obtain that the prebasis $G$ is a $\dO{1}$-basis  only if the coefficients of $G_2$ and $G_3$ satisfy the following equations:
    {\allowdisplaybreaks\begin{equation}\label{eq:conds final example}
    \begin{gathered}
    c_{2,1,1} c_{3,1,1}+c_{3,3,1}=0,\quad c_{2,1,1} c_{3,1,5}+c_{3,3,5}=0, \quad
    c_{2,1,1} c_{3,2,1}+c_{3,5,1}=0, \\ 
    c_{2,1,1} c_{3,2,2}+c_{3,5,2}=0, \quad c_{2,2,1} c_{3,1,1}+c_{3,5,1}=0, \quad c_{2,2,1} c_{3,1,5}+c_{3,5,5}=0, \\
    c_{2,2,1} c_{3,2,1}+c_{3,4,1}=0, \quad c_{2,2,1} c_{3,2,2}+c_{3,4,2}=0, \quad -c_{2,1,1} c_{3,2,1}+c_{2,2,1} c_{3,1,1}=0, \\
    c_{2,1,1} c_{3,1,2}-c_{2,1,2}+c_{3,3,2}=0, \quad c_{2,1,1} c_{3,2,3}-c_{2,1,2}+c_{3,5,3}=0, \quad c_{2,1,1} c_{3,1,3}-c_{2,1,3}+c_{3,3,3}=0, \\
    c_{2,1,1} c_{3,2,4}-c_{2,1,3}+c_{3,5,4}=0, \quad c_{2,1,1} c_{3,1,4}-c_{2,1,4}+c_{3,3,4}=0, \quad c_{2,1,1} c_{3,2,5}-c_{2,1,4}+c_{3,5,5}=0, \\
    c_{2,2,1} c_{3,1,2}-c_{2,2,2}+c_{3,5,2}=0,  \quad c_{2,2,1} c_{3,2,3}-c_{2,2,2}+c_{3,4,3}=0, \quad c_{2,2,1} c_{3,1,3}-c_{2,2,3}+c_{3,5,3}=0, \\
    c_{2,2,1} c_{3,2,4}-c_{2,2,3}+c_{3,4,4}=0, \quad c_{2,2,1} c_{3,1,4}-c_{2,2,4}+c_{3,5,4}=0, \quad c_{2,2,1} c_{3,2,5}-c_{2,2,4}+c_{3,4,5}=0, \\
    -c_{2,1,1} c_{3,2,5}+c_{2,2,1} c_{3,1,5}+c_{2,1,4}=0, \quad 
    -c_{2,1,1} c_{3,2,2}+c_{2,2,1} c_{3,1,2}-c_{2,2,2}=0, \\
    -c_{2,1,1} c_{3,2,3}+c_{2,2,1} c_{3,1,3}+c_{2,1,2}-c_{2,2,3}=0, \quad
    -c_{2,1,1} c_{3,2,4}+c_{2,2,1} c_{3,1,4}+c_{2,1,3}-c_{2,2,4}=0, \\
    -c_{2,1,2} c_{3,3,2}-c_{2,1,3} c_{3,5,2}-c_{2,1,4} c_{3,4,2}-c_{3,1,2}=0,\\   
    -c_{2,1,2} c_{3,3,3}-c_{2,1,3} c_{3,5,3}-c_{2,1,4} c_{3,4,3}-c_{3,1,3}=0,\\ 
    -c_{2,1,2} c_{3,3,4}-c_{2,1,3} c_{3,5,4}-c_{2,1,4} c_{3,4,4}-c_{3,1,4}=0,\\ 
    -c_{2,1,2} c_{3,3,5}-c_{2,1,3} c_{3,5,5}-c_{2,1,4} c_{3,4,5}-c_{3,1,5}=0,\\ 
    -c_{2,2,2} c_{3,3,2}-c_{2,2,3} c_{3,5,2}-c_{2,2,4} c_{3,4,2}-c_{3,2,2}=0,\\ 
    -c_{2,2,2} c_{3,3,3}-c_{2,2,3} c_{3,5,3}-c_{2,2,4} c_{3,4,3}-c_{3,2,3}=0,\\ 
    -c_{2,2,2} c_{3,3,4}-c_{2,2,3} c_{3,5,4}-c_{2,2,4} c_{3,4,4}-c_{3,2,4}=0,\\ 
    -c_{2,2,2} c_{3,3,5}-c_{2,2,3} c_{3,5,5}-c_{2,2,4} c_{3,4,5}-c_{3,2,5}=0,\\ 
    -c_{2,1,2} c_{3,3,1}-c_{2,1,3} c_{3,5,1}-c_{2,1,4} c_{3,4,1}+c_{2,1,1}-c_{3,1,1}=0,\\ 
    -c_{2,2,2} c_{3,3,1}-c_{2,2,3} c_{3,5,1}-c_{2,2,4} c_{3,4,1}+c_{2,2,1}-c_{3,2,1}=0.
    \end{gathered}
    \end{equation}}

    The family of ideals having a $\dO{1}$-basis  is in this case parametrized by the equations given in \eqref{eq:conds final example}. Indeed, we observe that in general once we consider values of the coefficients of $G_{\mindeg{\dO{1}}},\cdots,G_{t+1}$ that satisfy the equations of Theorem~\ref{effective char comm mat},\ref{it:effective char comm mat3} (in the present case equations~\eqref{eq:conds final example}), 
    we have that the ideal $J$ generated by $\bigcup_{i=\mindeg{\dO{1}}}^{t+1} G_i$ has border basis $G$, by imposing that  $G_d\subseteq J$, for $d\geq t+1$. This border basis of $J$ is unique by Proposition~\ref{prop: eu basis}. 
     We can explicitly compute the coefficients of the polynomials in $G_d$, for instance, using an iterative approach by solving the linear systems arising from the commutativity of formal multiplication matrices in degrees from $t+1$ to~$d$.
     
    For instance, for the following values of coefficients of the polynomials in $G_2$
     \begin{equation}\label{eq:fin ex c2}
     c_{2,1,1} = c_{2,1,4}=c_{2,2,1}=c_{2,2,2}=0,\; c_{2,1,2} = c_{2,2,3},\; c_{2,1,3} = c_{2,2,4},
     \end{equation}
      we obtain the values of the coefficients of $G_3$ by solving the linear system of the equations in~\eqref{eq:conds final example} after substituting~\eqref{eq:fin ex c2}:
     \begin{equation*}
     \begin{gathered}
     c_{3,1,1} = 0, \quad c_{3,1,2} = -c_{2,2,3}^{2},\quad c_{3,1,3} = -2 c_{2,2,3}c_{2,2,4}, \quad 
     c_{3,1,4} = -c_{2,2,4}^{2}, \quad c_{3,1,5} = 0,\\ 
     c_{3,2,1} = 0,\quad c_{3,2,2} = 0,\quad  c_{3,2,3} = -c_{2,2,3}^{2}, \quad c_{3,2,4} = -2 c_{2,2,3} c_{2,2,4}, \quad  c_{3,2,5} = -c_{2,2,4}^{2},\\
     c_{3,3,1} = 0, \quad c_{3,3,2} = c_{2,2,3}, \quad
     c_{3,3,3} = c_{2,2,4},\quad  c_{3,3,4} = 0, \quad c_{3,3,5} = 0,\\
     c_{3,4,1} = 0,\quad  c_{3,4,2} = 0,\quad c_{3,4,3} = 0,\quad
     c_{3,4,4} = c_{2,2,3},\quad c_{3,4,5} = c_{2,2,4},\\ c_{3,5,1} = 0,\quad
     c_{3,5,2} = 0, \quad c_{3,5,3} = c_{2,2,3},\quad c_{3,5,4} = c_{2,2,4},\quad c_{3,5,5} = 0.
    \end{gathered}
    \end{equation*}

\end{example}

\section{Future developments}

Despite their well-known strengths, border bases on a finite order ideal are available only for ideals of Krull dimension 0 in a non-homogeneous framework. In this paper we gave a notion of homogeneous border basis on an infinite order ideal, and adapted and generalized to this new framework a couple of characterizations (Theorems~\ref{char_rewrite_relations} and~\ref{char comm mat}) among those you can find in the literature for classical border bases.

A natural direction for further research is to consider also some other characterizations of classical border bases (e.g. reduction of $S$-polynomials, lifting of syzygies) and extend them to homogenous border bases. Moreover, it may be worthwhile to investigate numerical stability, and the eventual application to the study of Hilbert schemes, as it is done in the 0-dimensional case. Indeed, we believe that the characterizations already proved for homogeneous border bases will allow to parametrize families of ideals with a border basis on the same order ideal, and we expect that, up to some suitable hypotheses, these families embed as open subset in a Hilbert scheme, giving a new tool to explore the features of this mysterious geometric object.

\section*{Acknowledgments}
The first author is member of the INdAM group
GNSAGA.

\bibliographystyle{amsplain}
\bibliography{bibBorder}

\end{document}